\documentclass{birkmult}
\usepackage{amsmath,amsfonts} 

\usepackage[active]{srcltx}

\usepackage{color}

\newtheorem{thm}{Theorem}[section] 
\newtheorem{lem}[thm]{Lemma}     
\newtheorem{cor}[thm]{Corollary}

\newtheorem{rmk}[thm]{Remark}
\newtheorem{ex}[thm]{Example}
 \numberwithin{equation}{section}
\renewcommand{\Im}{\operatorname{Im}}

\newcommand{\abs}[1]{\left\vert#1\right\vert}
\newcommand{\diag}{\operatorname{diag}}

\newcommand{\R}{{\mathbb R}}
\newcommand{\Z}{{\mathbb Z}}

\newcommand{\C}{{\mathbb C}}
\newcommand{\N}{{\mathbb N}}
\newcommand{\Q}{{\mathbb Q}}
\newcommand{\eq} [1] {\begin{equation}\label{#1}\quad}
\newcommand{\en} {\end{equation}}

\title[Factorizations, Riemann-Hilbert problems, and corona]
 {Factorizations, Riemann-Hilbert problems\\ and the corona theorem} 

\author[C\^{a}mara]{M.C. C\^{a}mara}
\address{Departamento de Matem\'atica\\ Instituto Superior T\'ecnico\\
Av. Rovisco Pais\\ 1049-001 Lisboa, Portugal}

\email{cristina.camara@math.ist.utl.pt}

\author[Diogo]{C. Diogo}
\address{Departamento de M\'etodos Quantitativos\\
Lisbon University Institute (ISCTE-IUL)\\
Av. das For\c cas Armadas, 1649-026 Lisboa,
Portugal}

\email{cristina.diogo@iscte.pt}

\author[Karlovich]{Yu.I. Karlovich}
\address{Facultad de Ciencias\\ Universidad Aut\'onoma del Estado de
Morelos\\ Av. Universidad 1001, Col. Chamilpa,\\ C.P. 62209
Cuernavaca, Morelos,  M\'exico}

\email{karlovich@uaem.mx}

\author[Spitkovsky]{I.M. Spitkovsky}
\address{Department of Mathematics\\ College of William and Mary\\
Williamsburg, VA 23187\\ USA}

\email{ilya@math.wm.edu, imspitkovsky@gmail.com}


\thanks{The work was partially supported by FCT through
the Program POCI 2010/FEDER and the project PTDC/MAT/81385/2006, Portugal (the first two authors), by the SEP-CONACYT Project
No. 25564 and by PROMEP via ``Proyecto de Redes'', M\'exico (the third author), and by
the William \& Mary sabbatical research leave funding (the fourth author). The project PTDC/MAT/81385/2006 also partially supported
Yu.~Karlovich and I. Spitkovsky during their visits to Instituto Superior T\'ecnico.}

\begin{document}
\maketitle

\begin{abstract}
The solvability of the Riemann-Hilbert boundary value problem on the
real line is described in the case when its matrix coefficient admits a
Wiener-Hopf type factorization with bounded outer factors but rather
general diagonal elements of its middle factor. This covers, in
particular, the almost periodic setting, when the factorization multiples
belong to the algebra generated by the functions
$e_\lambda(x):=e^{i\lambda x}$, $\lambda\in\R$. Connections with
the corona problem are discussed. Based on those, constructive
factorization criteria are derived for several types of triangular
$2\times 2$ matrices with  diagonal entries $e_{\pm\lambda}$ and
non-zero off diagonal entry of the form $a_-e_{-\beta}+a_+e_\nu$ with
$\nu,\beta\geq 0$, $\nu+\beta>0$ and $a_\pm$ analytic and bounded
in the upper/lower half plane.
\end{abstract}

\section{Introduction}\label{introd}

The (vector) Riemann-Hilbert boundary value problem on the real line
$\R$ consists in finding two vector functions $\phi_\pm$, analytic in
the upper and lower half plane $\C^\pm=\{ z\in\C \colon \pm\Im
z>0\}$ respectively, satisfying the condition \eq{RHn}
\phi_-=G\phi_++g,\en imposed on their boundary values on $\R$.
Here $g$ is a given vector function and $G$ is a given matrix function
defined on $\R$, of appropriate sizes. It is well known that various
properties of (\ref{RHn}) can be described in terms of the (right) {\em
factorization} of its matrix coefficient $G$, that is, a representation of
$G$ as a product \eq{unt} G=G_-DG_+^{-1}, \en where $G_\pm$ and
their inverses are analytic in $\C^\pm$ and $D$ is a diagonal matrix
function with diagonal entries $d_j$ of a certain prescribed structure.
An exact definition of the factorization (\ref{unt}) is correlated with the
setting of the problem (\ref{RHn}), that is,  the requirements on the
boundary behavior of $\phi_\pm$.

To introduce a specific example, denote by $H_r^\pm$ the Hardy
classes in $\C^\pm$ and by $L_r$ the Lebesgue space on $\R$, with
$r\in (0,\infty]$. Let us also agree, for any set $X$, to denote by $X^n$
($X^{n\times n}$) the set of all $n$-vectors (respectively, $n\times n$
matrices) with entries in $X$.

With this notation at hand, recall that the {\em $L_p$ setting} of
(\ref{RHn}) is the one for which $g\in L_p^n$ and $\phi_\pm\in
(H_p^\pm)^n$. An appropriate representation (\ref{unt}), in this
setting with $p>1$, is the  so called $L_p$ factorization of $G$: the
representation (\ref{unt}) in which \eq{pcond}
\lambda_\pm^{-1}G_\pm\in (H_p^\pm)^{n\times n},\quad
\lambda_\pm^{-1}G_\pm^{-1}\in (H_q^\pm)^{n\times n}\text{ and }
d_j=(\lambda_-/\lambda_+)^{\kappa_j}. \en Here
\[ 1<p<\infty,\quad q=\frac{p}{p-1},\  \lambda_\pm
(z)=z\pm i,
\] and the integers $\kappa_j$ are called the (right) {\em partial indices} of $G$.

A full solvability picture for the problem (\ref{RHn}) with $L_p$
factorable $G$ can be extracted from \cite[Chapter 3]{LS}, see also
\cite{GKS03}. The central result in this direction is (the real line version
of) the Simonenko's theorem, according to which (\ref{RHn}) has a
unique solution for every right hand side $g$ --- equivalently, the
associated Toeplitz operator $T_G=: P_+G| (H_p^+)^n$ is invertible ---
if and only if $G$ admits an $L_p$ factorization (\ref{unt}) with $D=I$,
subject to the additional condition \eq{Phi} G_-P_+G_-^{-1} \text{ is a
densely defined bounded operator on } L^n_p.\en Here $P_+$ is the
projection operator of $L_p$ onto $H_p^+$ along $H_p^-$, defined on
vector (or matrix) functions entrywise.

In this paper, we take particular  interest in {\em bounded}
factorizations for which  in (\ref{unt}), by definition,  \eq{bf} G_+^{\pm
1}\in (H_\infty^+)^{n\times n}, \quad G_-^{\pm 1}\in
(H_\infty^-)^{n\times n}. \en Of course, with $d_j$ as in (\ref{pcond}) a
bounded factorization of $G$ is its $L_p$ factorization simultaneously
for all $p\in (1,\infty)$, and the additional condition (\ref{Phi}) is
satisfied. However, some meaningful conclusions regarding the
problem (\ref{RHn}) can be drawn from the relation (\ref{unt})
satisfying (\ref{bf}) even without any additional information about the
diagonal entries of $D$. This idea for $L_p$ factorization on closed
curves was first discussed in \cite{Spit881}; in Section~\ref{RHP} we
give a detailed account of the bounded factorization version. That
includes in particular the interplay between the factorization problem
and the corona theorem.

Section~\ref{s:AP} deals with the {\em almost periodic} ($AP$ for
short) setting, in which the elements of the matrix function involved
belong to the algebra $AP$ generated by the functions \eq{el}
e_\lambda(x)=e^{i\lambda x},\quad \lambda\in\R,\en the diagonal
elements $d_j$ being chosen among its generators $e_\lambda$. In
this case not only we consider the solvability of (\ref{RHn}) when $G$
admits an $AP$ factorization, but also address the converse question:
what information on the existence and the properties of that
factorization can be obtained from a solution to a homogeneous
problem \eq{RH} G\phi_+=\phi_-,\quad \phi_\pm\in (H_p^\pm)^n \en
 with $p=\infty$.

In Sections~\ref{applications}, \ref{GAP} we consider classes of matrix
functions $G$ for which (\ref{RHn}) is closely related with a
convolution equation on an interval of finite length. By determining a
solution to the homogeneous Riemann-Hilbert problem (\ref{RH}) in
$H^\pm_\infty$ and applying the results of the previous sections, we
study the factorability of $G$ and the properties of the related Toeplitz
operator $T_G$. In particular, invertibility conditions for this operator
are obtained and a subclass of matrix functions is identified for which
invertibility of $T_G$ is (somewhat surprisingly) equivalent to its
semi-Fredholmness.

\section{Riemann-Hilbert problems and factorization}\label{RHP}

We start with the description of the solutions to (\ref{RHn}), in terms of
a bounded factorization (\ref{unt}).

\begin{thm}\label{th:A}
Let $G$ admit a bounded factorization {\em (\ref{unt})}. Then all
solutions of the problem {\em (\ref{RHn})} satisfying $\phi_\pm\in
(H_p^\pm)^n$ for some $p\in [1,\infty]$ are given by \eq{sol}
\phi_+=\sum_j \psi_jg_j^+,\quad \phi_-=\sum_j d_j\psi_jg_j^-+g.\en
Here $g_j^\pm$ stands for the $j$-th column of $G_\pm$:
\eq{Gmp}G_-=\left[
\begin{array}{cccc}g_1^- & g_2^- & \ldots & g_n^- \\
\end{array}\right], \quad G_+=\left[\begin{array}{cccc}
 g_1^+ & g_2^+ & \ldots & g_n^+ \\\end{array}\right],\en and
$\psi_j$ is an arbitrary function satisfying \eq{scRH} \psi_j\in
H_p^+,\quad d_j\psi_j+(G_-^{-1}g)_j\in H_p^-.\en
\end{thm}
In other words, the Riemann-Hilbert problem (\ref{RHn}) with a matrix
coefficient $G$ admitting a bounded factorization can be untangled
into $n$ scalar Riemann-Hilbert problems, in the same $L_p$ setting.

The proof of Theorem~\ref{th:A} is standard in the factorization theory,
based on a simple change of unknowns $\phi_\pm=G_\pm\psi_\pm$.
We include it here for completeness.
\begin{proof} If $(\phi_+,\phi_-)$ is a solution to (\ref{RHn}), then defining
$\psi:=(\psi_j)_{j=1,\ldots,n}=G_+^{-1}\phi_+$ we get
$\phi_+=G_+\psi$, $\phi_-=G_-D\psi+g$, which is equivalent to
(\ref{sol}), and (\ref{scRH}) is satisfied. Conversely, if (\ref{scRH})
holds for all $j=1,\ldots,n$, then $\phi_+=G_+\psi\in (H_p^+)^n$,
$\phi_-=G_-D\psi+g\in (H_p^-)^n$, and (\ref{RHn}) holds.
\end{proof}

We will say that a function $f$, defined a.e. on $\R$, is of {\em
non-negative type} if \eq{dj1} f\in H_\infty^+ \text{ or } f^{-1}\in
H_\infty^-.\en The type is {\em non-positive} if \eq{dj2} f\in H_\infty^-
\text{ or } f^{-1}\in H_\infty^+, \en (strictly) {\em positive} if (\ref{dj1})
holds while (\ref{dj2}) does not, and {\em neutral} if both (\ref{dj1}),
(\ref{dj2}) hold.

\begin{lem}\label{l:uniq}For $d_j$ of positive type, there is at most one
function $\psi_j$ satisfying {\em (\ref{scRH})}. \end{lem}
\begin{proof}It suffices to show that the only function $\psi\in H_p^+$ satisfying
$d_j\psi\in H_p^-$ is zero.

If the first condition in (\ref{dj1}) holds for $f=d_j$, then $d_j\psi\in
H_p^+$ simultaneously with $\psi$ itself. From here and $d_j\psi\in
H_p^-$ it follows that $d_j\psi$ is a constant. If this constant is
non-zero (which is only possible if $p=\infty$), then $d_j$ is invertible
in $H_\infty^+$ which contradicts the strict positivity of its type. On the
other hand, the product $d_j\psi$ of two analytic functions may be
identically zero only if one of them is. It cannot be $d_j$ (once again,
since otherwise the first condition in (\ref{dj2}) would hold); thus,
$\psi=0$.

The second case of (\ref{dj1}) can be treated in a similar way.
\end{proof}

As an immediate consequence we have:
\begin{cor}\label{cor:B} If $G$ admits a bounded factorization  with all $d_j$
of positive type, then the homogeneous Riemann-Hilbert problem
{\em (\ref{RH})} has only the trivial solution $\phi_+=\phi_-=0$ for any
$p\in [1,\infty]$.
\end{cor}

If $d_j$ is of neutral type, then by definition it is either invertible in
$H_\infty^+$, or in $H_\infty^-$, or is equal to zero. Disallowing the
latter case, and absorbing $d_j$ in the column $g_j^\pm$ in the
former, we may without loss of generality suppose that all such $d_j$
are actually equal 1. With this convention in mind, the following result
holds.

\begin{cor}\label{cor:C} Let  $G$ admit a bounded factorization with all $d_j$
of non-negative type, $d_j\neq 0$. Then the homogeneous problem
{\em (\ref{RH})} for $1\leq p<\infty$ has only the trivial solution, and
for $p=\infty$ all its solutions are given by \eq{hs} \phi_+=\sum_{j\in
J} c_jg_j^+,\quad \phi_-=\sum_{j\in J} c_jg_j^-. \en Here $g_j^\pm$
are as in {\em (\ref{Gmp})}, $c_j\in\C$, and $j\in J$ if and only if $d_j$
is of neutral type. \end{cor}
\begin{proof}From (\ref{sol}) and from Lemma~\ref{l:uniq} we have \[
\phi_+=\sum_{j\in J}\psi_jg_j^+,\quad \phi_-=\sum_{j\in
J}d_j\psi_jg_j^-,\] while our convention regarding the neutral type
allows us to drop the functions $d_j$ in the expression for $\phi_-$.
Finally, (\ref{scRH}) with $d_j$ of neutral type and $g=0$ means that
$\psi_j\in H_p^+\cap H_p^-$, and thus $\psi_j$ is a constant ($=0$ if
$p<\infty$). \end{proof}

Recall that the factorization (\ref{unt}) is {\em canonical} if the middle
factor $D$ of it is the identity matrix, and can therefore be dropped:
\eq{canfac} G=G_-G_+^{-1}.\en The following criterion for bounded
canonical  factorability is easy to establish, and actually well known.
We state it here, with proof, for the sake of completeness and ease of
references.

\begin{lem}\label{l:canfac}$G$ admits a bounded canonical factorization {\em (\ref{canfac})}
if and only if problem {\em (\ref{RH})} with $p=\infty$ has solutions
$\phi_j^\pm$, $j=1,\ldots, n$, such that \eq{ind} \det
[\phi_1^\pm\ldots \phi_n^\pm] \text{ is invertible in }
H_\infty^\pm.\en  If this is the case, then one of the factorizations is
given by \eq{phiG} G_\pm =\left[\phi_1^\pm\ldots \phi_n^\pm\right],
\en and all solutions to {\em (\ref{RH})} in $H_\infty^\pm$ are linear
combinations of $\phi_j^\pm$.
\end{lem}
\begin{proof}If (\ref{canfac}) holds with $G_\pm$ satisfying (\ref{bf}), then one may
choose $\phi_j^\pm$ as the $j$-th column of $G_\pm$. Conversely, if
$\phi_j^\pm$ satisfy (\ref{RH}) and (\ref{ind}), then $G_\pm$ given by
(\ref{phiG}) satisfy $GG_+=G_-$ and (\ref{bf}). Therefore, (\ref{canfac})
holds and delivers a bounded canonical factorization of $G$.

The last statement now follows from Corollary~\ref{cor:C}.
\end{proof}

Observe that for $G$ with constant non-zero determinant, the
determinants of matrix functions $G_\pm$ given by (\ref{phiG})  also
are necessarily constant. So, (\ref{ind}) holds if and only if the vector
functions $\phi_1^+(z),\ldots,\phi_n^+(z)$  (or
$\phi_1^-(z),\ldots,\phi_n^-(z)$) are linearly independent for at least
one value of $z\in\C^+$ (resp., $\C^-$).

As it happens, if $G$ admits a bounded canonical factorization, all its
bounded factorizations (with no a priori conditions on $d_j$) are
forced to be ``almost'' canonical. The precise statement is as follows.

\begin{thm}\label{th:D}Let $G$ have a bounded canonical factorization
$G=\widetilde{G}_-\widetilde{G}_+^{-1}$. Then all its bounded
factorizations are given by {\em (\ref{unt})}, where each $d_j$ has a
bounded canonical factorization \eq{d} d_j=d_{j-}d_{j+}^{-1}, \quad
j=1,\ldots,n,\en   \eq{GD} G_\pm=\widetilde{G}_\pm Z
D_\pm^{-1},\quad D_\pm=\diag [ d_{1\pm},\ldots, d_{n\pm}],\en and
$Z$ is an arbitrary invertible matrix in $\C^{n\times n}$.
\end{thm}
\begin{proof}Equating two factorizations $\widetilde{G}_-\widetilde{G}_+^{-1}$ and $G_-DG_+^{-1}$ yields
\eq{D} D=
G_-^{-1}\widetilde{G}_-\widetilde{G}_+^{-1}G_+=F_-F_+^{-1},\en where
$F_\pm,\ F_\pm^{-1}\in (H^\pm_\infty)^{n\times n}$. Consequently,
$D$ admits a bounded canonical factorization, and therefore the
Toeplitz operator $T_D$ is invertible on $(H_p^+)^n$ for $p\in
(1,\infty)$. Being the direct sum of $n$ scalar Toeplitz operators
$T_{d_j}$, this implies that each of the latter also is invertible, on
$H_p^+$. Thus, each of the scalar functions $d_j$ admits a canonical
$L_p$ factorization. Let  (\ref{d}) be such a factorization,
corresponding\footnote{The interpolation property of factorization
\cite[Theorem 3.9]{LS} implies that in our setting the canonical $L_p$
factorization of $d_j$ is the same for all $p\in (1,\infty)$ but this fact
has no impact on the reasoning.} to $p=2$. Then, according to
(\ref{D}) the elements $f_{ij}^\pm$ of the matrix functions $F_\pm$ are
related as $f_{ij}^-=d_jf_{ij}^+$. Due to the invertibility of $F_\pm$, for
each $j$ the functions $f^\pm_{ij}$ are non-zero for at least one value
of $i$. Choosing such $i$ arbitrarily, and abbreviating the respective
$f_{ij}^\pm$ simply to $f_{j\pm}$, we have \[
f_{j-}d_{j-}^{-1}=f_{j+}d_{j+}^{-1}.\]  The left and right hand side of the
latter equality is a function in $\lambda_-H_2^-$ and
$\lambda_+H_2^+$, respectively. Hence, each of them is just a scalar
(non-zero, due to our choice of $i$). So, $d_{j\pm}\in H_\infty^\pm$.

Letting $d_\pm=\prod_{j=1}^n d_{j\pm}$,  from here we obtain that
$\det D=d_-d_+^{-1}$, with $d_\pm\in H_\infty^\pm$. But (\ref{D})
implies also that $\det D$ admits the bounded analytic factorization
$\det F_-/\det F_+$. Thus, \[ d_-/\det F_-=d_+/\det F_+,\] with left/right
hand side lying in $H_\infty^\pm$, respectively. Hence, $d_\pm$
differs from $\det F_\pm$ only by a (clearly, non-zero) scalar multiple,
and therefore is invertible in $H_\infty^\pm$. This implies the
invertibility of each multiple $d_{j\pm}$ in $H_\infty^\pm$, $j=1,\ldots,
n$, so that each representation (\ref{d}) is in fact  a bounded
canonical factorization.

With the notation $D_\pm$ as in (\ref{GD}), the first equality in
(\ref{D}) can be rewritten as \[
\widetilde{G}_-^{-1}G_-D_-=\widetilde{G}_+^{-1}G_+D_+.\] Since the
left/right hand side is invertible in $(H_\infty^\mp)^{n\times n}$, each
of them is in fact an invertible constant matrix $Z$. This implies the
first formula in (\ref{GD}).\end{proof} According to (\ref{GD}) with
$D=I$, two bounded canonical factorizations of $G$ are related as
\eq{descan} G_\pm=\widetilde{G}_\pm Z, \text{ where } Z\in
\C^{n\times n}, \det Z\neq 0,\en --- a well-known fact.

When $n=2$, the results proved above simplify in a natural way. We
will state only one such simplification, once again, for convenience of
references.

\begin{thm}\label{th:E}Let $G$ be a $2\times 2$ matrix function admitting a bounded factorization
{\em (\ref{unt})} with one of the diagonal entries (say $d_2$) of
positive type. Then the problem {\em (\ref{RH})} has non-trivial
solutions in $H_p^\pm$ for some $p\in [1,\infty]$ if and only if $d_1$
admits a representation $d_1=d_{1-}d_{1+}^{-1}$ with $d_{1\pm}\in
H_p^\pm$. If this condition holds, then all the solutions of {\em
(\ref{RH})} are given by
\[ \phi_+=\psi g_1^+,\quad \phi_-=d_1\psi g_1^-, \] where $g_1^\pm$ is the first column of
$G_\pm$ in the factorization {\em (\ref{unt})} and $\psi\in H_p^+$ is
an arbitrary function satisfying  $d_1\psi\in H_p^-$.\end{thm}
\begin{proof}{\sl Sufficiency.} If $d_1=d_{1-}d_{1+}^{-1}$ with $d_{1\pm}\in H_p^\pm$, then obviously $d_{1+}\neq 0$
and \[ \phi_+=d_{1+}g_1^+, \quad \phi_-=d_{1-}g_1^- \] is a non-trivial
solution to (\ref{RH}).

{\sl Necessity.} By Lemma~\ref{l:uniq} and Theorem~\ref{th:A} the
solution must be of the form $\phi_+=\psi g_1^+$, $\phi_-=d_1\psi
g_1^-$ with $\psi\in H_p^+\setminus\{0\}$, $d_1\psi\in H_p^-$. It
remains to set $d_{1+}=\psi$, $d_{1-}=d_1\psi$. \end{proof}

More interestingly, there is a close relation between  factorization and
corona problems.

Recall that  a vector function $\omega$ with entries
$\omega_1,\ldots,\omega_n\in H_\infty^+$ satisfies the {\em corona
condition} in $\C^+$ (notation: $\omega\in CP^+$) if and only if \[
\inf_{z\in\C^+}(|\omega_1(z)|+\cdots+|\omega_n(z)|)>0.\] The
\emph{corona condition in $\C^-$} for a vector function $\omega\in
(H_\infty^-)^n$ and the notation $\omega\in CP^-$ are introduced
analogously.

By the corona theorem, $\omega\in CP^\pm$ if and only if there exists
$\omega^*=(\omega_1^*,\ldots,\omega_n^*)\in (H_\infty^\pm)^n$
such that \[
\omega_1\omega_1^*+\cdots+\omega_n\omega_n^*=1.\]

\begin{thm}\label{th:F} If an $n\times n$ matrix function $G$ admits
a bounded canonical factorization, then any non-trivial solution of
problem {\em (\ref{RH})} in $(H_\infty^\pm)^n$ actually lies in
$CP^\pm$. \end{thm}
\begin{proof}Let $G$ admit a bounded canonical factorization (\ref{canfac}).
By Corollary~\ref{cor:C}, every non-trivial solution $\phi_\pm$ of
(\ref{RH})  is a nontrivial linear combination of the columns
$g_j^\pm$, $j=1,\ldots,n$. According to (\ref{descan}), any such
combination, in turn, can be used as a column of some (perhaps,
different) bounded canonical factorization of $G$. Being a column of
an invertible element of $(H_\infty^\pm)^{n\times n}$, it must lie in
$CP^\pm$.\end{proof}

The following result is a somewhat technical generalization of
Theorem~\ref{th:F}, which will be used later on.

\begin{thm}\label{th:G}
Let $G$ be an $n\times n$ matrix function admitting a bounded
factorization  {\em (\ref{unt})} in which for all $k=2,\ldots,n$ either
$d_k=d_1\neq 0$ or $d_1^{-1}d_k$ is a function of positive type. Then
for any pair of non-zero vector functions $\phi_\pm\in
(H^\pm_\infty)^n$ satisfying $G\phi_+=\phi_-$, $d_1\phi_+\in
(H_\infty^+)^n$,  in fact stronger conditions \eq{2.19} d_1\phi_+\in
CP^+,\quad \phi_-\in CP^-\en hold. In order for such pairs to exist,
$d_1$ has to be of non-positive type.
\end{thm}

\begin{proof}Let $\widetilde{G}=d_1^{-1}G$. Then, due to (\ref{unt}),
\eq{2.20} \widetilde G=G_-\widetilde D G_+^{-1}\text{ with } \widetilde
D=\diag [1,d_1^{-1}d_2,\ldots,d_1^{-1}d_n],  \en which of course is a
bounded factorization of $\widetilde{G}$.

Condition $\phi_-=G\phi_+$ implies that
$\phi_-=\widetilde{G}d_1\phi_+$, so that  the pair $d_1\phi_+,\phi_-$
is a non-trivial solution of the homogeneous Riemann-Hilbert problem
with the coefficient $\widetilde{G}$.

If $d_1=d_2=\cdots=d_n$, then (\ref{2.20}) delivers a bounded
canonical factorization of $\widetilde{G}$, so that the desired result
follows from Theorem~\ref{th:F}. If, on the other hand,
$d_1^{-1}d_2,\ldots,d_1^{-1}d_n$ are all of positive type, then
$d_1\phi_+$ and $\phi_-$ differ only by a (non-zero) constant scalar
multiple from the first column of $G_+$ and $G_-$ respectively,
according to Corollary~\ref{cor:C}. This again implies (\ref{2.19}).

Finally, from $d_1\phi_+\in CP^+$ and $\phi_+\in (H^+_\infty)^n$ it
follows that $d_1^{-1}\in H^+_\infty$, that is, $d_1$ is of non-positive
type.
\end{proof}

The exact converse of Theorem~\ref{th:F} is not true. However, a
slightly more subtle result holds.
\begin{thm}\label{th:FA}Let $G\in L_\infty^{2\times 2}$ be such that there exists
a solution of problem {\em (\ref{RH})} in $CP^\pm$. Then the Toeplitz
operators $T_G$ on $(H_p^+)^2$ and $T_{\det G}$ on $H_p^+$ are
Fredholm only simultaneously, and their defect numbers
coincide.\end{thm}
\begin{proof}The existence of the above mentioned solutions implies (see, e.g.,
computations in \cite[Section 22.1]{BKS1}) that \[ G=
X_-\left[\begin{matrix}\det G & 0 \\ * & 1\end{matrix}\right] X_+,\]
where $X_\pm$ is an invertible element of $(H_\infty^\pm)^{2\times
2}$. From here and elementary  properties of block triangular
operators it follows that the respective defect numbers (and thus the
Fredholm behavior) of $T_G$ and $T_{\det G}$ are the same.
\end{proof} According to Theorem~\ref{th:FA}, in the particular case
when $\det G$ admits a canonical factorization,  the operator $T_G$ is
invertible provided that (\ref{RH}) has a solution in $CP^\pm$. For
$\det G\equiv 1$ the latter result was essentially established in
\cite{BasKS03}.  An alternative, and more detailed, proof of
Theorem~\ref{th:FA} can be found in \cite{CDR}, Theorems~4.1 and
4.4.

Let now $\mathcal B$ be a subalgebra of $L_\infty$ (not necessarily
closed in $L_\infty$ norm) such that, for any $n$, a matrix function
$G\in {\mathcal B}^{n\times n}$ admits a bounded canonical
factorization if and only if the operator $T_G$ is invertible in
$(H_p^+)^n$ for at least one (and therefore all) $p\in (1,\infty)$. There
are many classes satisfying this property, e.g., decomposable algebras
of continuous functions (see \cite{CG,LS}) or the algebra $APW$
considered below.

\begin{thm}\label{th:2.7A}
Let $G\in {\mathcal B}^{2\times 2}$ with $\det G$ admitting a
bounded canonical factorization, and let $\phi_\pm\in
(H_\infty^\pm)^2$ be a non-zero solution to {\em (\ref{RH})}. Then
$G$ has a bounded canonical factorization if and only if $\phi_\pm\in
CP^\pm$.
\end{thm}
\begin{proof}Necessity follows from Theorem~\ref{th:F} and sufficiency from Theorem~\ref{th:FA}.
The latter can also be deduced from \cite[Theorem 3.4]{BasKS03}
formulated there for $G$ with constant determinant but  remaining
valid if $\det G$ merely admits a bounded canonical factorization.
\end{proof}

\section{{$AP$} factorization}\label{s:AP}
We will now recast the results of the previous section in the
framework of $AP$ factorization. To this end, recall that $AP$ is the
uniform closure of all linear combinations $\sum c_je_{\lambda_j}$,
$c_j\in\C$,  with $e_{\lambda_j}$ defined by (\ref{el}), while these
linear combinations themselves form the set $APP$ of all {\em almost
periodic polynomials}. Properties of $AP$ functions are discussed in
detail in \cite{C,Le}, see also \cite[Chapter 1]{BKS1}. In particular, for
every $f\in AP$ there exists its {\em mean value}
\[ M(f)=\lim_{T\to\infty}\frac{1}{2T}\int_{-T}^T f(t)\, dt.\] This yields the existence of
$\widehat{f}(\lambda):=M(e_{-\lambda}f)$, the {\em Bohr-Fourier
coefficients} of $f$.  For any given $f\in AP$, the set
\[ \Omega(f)=\{\lambda\in\R\colon \widehat{f}(\lambda)\neq 0\}
\] is at most countable, and is called the {\em Bohr-Fourier spectrum}
of $f$.  The formal {\em Bohr-Fourier series}
$\sum_{\lambda\in\Omega(f)}\widehat{f}(\lambda)e_\lambda$ may
or may not converge; we will write $f\in APW$ if it does converge
absolutely. The algebras $AP$ and $APW$ are inverse closed in
$L_\infty$; moreover, for an invertible $f\in AP$ there exists an
(obviously, unique) $\lambda\in\R$ such that a continuous branch of
$\log(e_{-\lambda}f)\in AP$. This value of $\lambda$ is called the
{\em mean motion} of $f$; we will denote it $\kappa(f)$.

Finally, let  \[ AP^\pm=\{f\in AP\colon\Omega(f)\subset\R_\pm\},\]
where of course $\R_\pm=\{ x\in\R\colon\pm x\geq 0\}$. Denote also
\[ APW^\pm = AP^\pm\cap APW,\quad APP^\pm = AP^\pm\cap APP.\] Clearly,
\[ APP^\pm \subset APW^\pm\subset AP^\pm\subset H_\infty^\pm. \]

An $AP$ factorization of $G$, by definition, is a representation
(\ref{unt}) in which $G_\pm$ are subject to the conditions \eq{AP}
G_+^{\pm 1}\in (AP^+)^{n\times n}, \quad G_-^{\pm 1}\in
(AP^-)^{n\times n},\en more restrictive than (\ref{bf}), and the diagonal
entries of $D$ are of the form $d_j=e_{\delta_j}$, $j=1,\ldots,n$. The
real numbers $\delta_j$ are called the (right) partial $AP$ indices of
$G$,  and by an obvious column permutation in $G_\pm$ we may
assume that they are arranged in a non-decreasing order:
$\delta_1\leq\delta_2\leq\dots\leq\delta_n$.

A particular case of $AP$ factorization occurs when conditions
(\ref{AP}) are changed to more restrictive ones: \[ G_+^{\pm 1}\in
(APW^+)^{n\times n}, \quad G_-^{\pm 1}\in (APW^-)^{n\times n},\] or
even
\[ G_+^{\pm 1}\in
(APP^+)^{n\times n}, \quad G_-^{\pm 1}\in (APP^-)^{n\times n}.\]

These are naturally called $APW$ and $APP$ factorization of $G$,
respectively. Of course, $G$ has to be an invertible element of
$AP^{n\times n}$ ($APW^{n\times n}$, $APP^{n\times n}$) in order to
admit an $AP$ (resp., $APW$, $APP$) factorization. Moreover, the
partial $AP$ indices of $G$ should then add up to the mean motion of
its determinant: \eq{mm} \delta_1+\cdots +\delta_n=\kappa(\det
G),\en as can be seen by simply taking determinants of both sides.

 All the statements of Section~\ref{RHP} are valid in these settings,
and some of them can even be simplified. For instance,  a diagonal
element of $D$ is of positive, negative or neutral type  (in the sense of
definitions (\ref{dj1}), (\ref{dj2})) if and only if  the corresponding
partial $AP$ index $\delta_j$ is respectively positive, negative or equal
zero.

Corollary~\ref{cor:C}, for example, applies to $AP$-factorable matrix
functions $G$ with non-negative partial $AP$ indices. Formulas
(\ref{hs}) imply then that all solutions of (\ref{RH}) in
$(H_\infty^\pm)^n$ are automatically in $(AP^\pm)^n$ (and even
$(APW^\pm)^n$ or $(APP^\pm)^n$, provided that $G$ is respectively
$APW$- or $APP$-factorable).

Lemma~\ref{l:canfac} takes the following form.
\begin{thm}\label{th:H}
An $n\times n$ matrix function $G$ admits a canonical $AP$ ($APW$)
factorization if and only if  there exist  $n$ solutions
$(\psi_j^+,\psi_j^-)$ to {\em (\ref{RH})} in $(AP^\pm)^n$ (resp.,
$(APW^\pm)^n$), such that $\det [\psi_1^\pm\ldots \psi_n^\pm]$ are
bounded from zero in $\C^\pm$.
\end{thm}
The respective criterion  for $APP$ factorization is slightly different,
because $APP^\pm$, as opposed to $AP^\pm$ and $APW^\pm$, are
not inverse closed in $H_\infty^\pm$. Moreover, the only invertible
elements of $APP^\pm$ are non-zero constants. Therefore, we arrive
at
\begin{cor}\label{cor:H}
An $n\times n$ matrix function $G$ admits a canonical $APP$
factorization if and only if  there exist  $n$ solutions
$(\psi_j^+,\psi_j^-)$ to {\em (\ref{RH})} in $(APP^\pm)^n$ with
constant non-zero $\det [\psi_1^\pm\ldots \psi_n^\pm]$.
\end{cor}

Similarly to the case in Section~\ref{RHP}, for matrix functions $G$
with constant determinant the condition on $\det [\psi_1^\pm\ldots
\psi_n^\pm]$ holds whenever at least one of them is non-zero at just
one point of $\C^\pm\cup\R$. All non-trivial solutions  to (\ref{RH})
are actually in $CP^\pm$, as guaranteed by Theorem~\ref{th:F}.

{Theorem~\ref{th:A} of course remains valid when $G$ admits an $AP$
factorization; the only change needed  is that $d_j$ in formulas
(\ref{sol}), (\ref{scRH}) should be substituted by $e_{\delta_j}$. For the
homogeneous problem (\ref{RH}) this yields the following.
\begin{thm}\label{th:hap}Let $G$ admit an $AP$ factorization {\em (\ref{unt})}. Then
the general solution of problem {\em (\ref{RH})} in
$(H_\infty^\pm)^n$  is given by \eq{hap}  \phi_+=\sum_j
\psi_jg_j^+,\quad \phi_-=\sum_j e_{\delta_j}\psi_jg_j^-, \en where the
summation is with respect to those $j$ for which $\delta_j\leq 0$,
$\psi_j$ are constant whenever $\delta_j=0$ and satisfy
\eq{conap}\psi_j\in H_\infty^+\cap e_{-\delta_j}H^-_\infty \text{
whenever } \delta_j<0. \en
\end{thm} Observe that $\phi_\pm$ given by (\ref{hap}) belong to $AP^n$ if and only if condition
(\ref{conap}) is replaced by a more restrictive \[ \psi_j\in AP, \quad
\Omega(\psi_j)\subset [0,-\delta_j] \] (where by convention $\psi_j=0$
if $\delta_j>0$), since \eq{conapA}
\psi:=(\psi_j)=G_+^{-1}\phi_+=D^{-1}G_-^{-1}\phi_-.\en

Moreover, if in fact $G$ is $APW$ factorable, then the functions
(\ref{hap}) are in $APW^n$ if and only if
\[ \psi_j\in APW, \quad \Omega(\psi_j)\subset [0,-\delta_j]. \]  Solutions of (\ref{RH}) in $(H_\infty^\pm)^n$ are automatically
in $AP$ ($APW$) if $G$ is $AP$- (resp., $APW$-) factorable with
non-negative partial $AP$ indices, since in this case $D^{-1}\in APP^-$
and (\ref{conapA}) implies that $\psi\in\C^n$.  On the other hand, if
$G$ is $APW$ factorable with at least one negative partial $AP$ index,
then all three classes are distinct. Indeed, for any $j$ corresponding to
$\delta_j<0$ there is a plethora of functions $\psi_j$ satisfying
(\ref{conap}) not lying in $AP$, as well as functions in $AP\setminus
APW$ with the Bohr-Fourier spectrum in $[0,-\delta_j]$.

The case of exactly one non-positive partial $AP$ index is of special
interest.

\begin{cor}\label{th:I} Let $G$ admit an $AP$ factorization with the partial $AP$ indices
$\delta_1\leq 0<\delta_2\leq\cdots$. Then all solutions to {\em
(\ref{RH})} in $(H_\infty^+)^n$ ($AP^n$, $APW^n$) are given by
\eq{2.23} \phi_+=fg_1^+,\quad \phi_-=e_{\delta_1}fg_1^-,\en where $f$
is an arbitrary $H_\infty^+$ function such that $e_{\delta_1}f\in
H^-_\infty$ (resp., $f\in AP$ or $f\in APW$ and $\Omega(f)\subset
[0,-\delta_1]$). \end{cor} }

For $n=2$ the reasoning of Theorem~\ref{th:G} suggests an
appropriate modification of (\ref{RH}) for which some solutions are
forced to lie in $AP$.  Recall our convention $\delta_1\leq\delta_2$
according to which the condition on $d_1,d_2$ in Theorem~\ref{th:G}
holds automatically.

\begin{thm}\label{th:J}
Let $G$ be a $2\times 2$ $AP$ factorable matrix function with  partial
indices $\delta_1$, $\delta_2$ ($\delta_1\leq \delta_2$). Then any
non-zero pair $(\phi_+,\phi_-)$ with $\phi_+\in (H_\infty^+)^2\cap
e_{-\delta_1}(H_\infty^+)^2$, $\phi_-=G\phi_+\in (H_\infty^-)^2$
satisfies \[ \phi_\pm\in (AP^\pm)^2,\quad e_{\delta_1}\phi_+\in
CP^+,\quad \phi_-\in CP^-, \] and in order for such pairs to exist it is
necessary and sufficient that $\delta_1\leq 0$. If $\delta_2>\delta_1$,
all those solutions have the form \[ \phi_+=ce_{-\delta_1}g_1^+\,,\;\;
\phi_-=cg_1^- \;\; \text{with} \;\; c\in\C\setminus\{0\}. \] For
$\delta_2=\delta_1$, $\phi_+$ and $\phi_-$ are the same non-trivial
linear combinations of the columns of $e_{-\delta_1}G_+$ and $G_-$.
\end{thm}

Of course, Theorem~\ref{th:J} holds with $AP$ changed to $APW$ or
$APP$ everywhere in its statement.

Recall that a Toeplitz operator with scalar $AP$ symbol $f$ is
Fredholm on $H_p^+$ for some (equivalently: all) $p\in (1,\infty)$ if
and only if it is invertible if and only if $f$ is invertible in $AP$ with
mean motion zero. Therefore, Theorem~\ref{th:FA} implies
\begin{lem}\label{l:T}Let $G\in AP^{2\times 2}$ be such that there exists a solution
of {\em (\ref{RH})} in $CP^\pm$. Then the Toeplitz operator $T_G$ is
invertible on $(H_p^+)^2$, $1<p<\infty$, if and only if $\kappa(\det
G)=0$. \end{lem}

Passing to the $APW$ setting,  we invoke the result according to
which $T_G$ with $G\in APW^{n\times n}$ is invertible if and only if
$G$ admits a canonical $AP$ (or $APW$) factorization.
Lemma~\ref{l:T} then implies (compare with Theorem~\ref{th:2.7A}):

\begin{thm}\label{th:cocri} Let $G\in APW^{2\times 2}$. Then $G$ admits a canonical $AP$
factorization if and only if $\kappa(\det G)=0$ and problem {\em
(\ref{RH})} has a solution in $CP^\pm$. If this is the case, then every
non-zero solution of {\em (\ref{RH})} is in $(APW^\pm)^2\cap
CP^\pm$. \end{thm} The first part of Theorem~\ref{th:cocri} for $G$
with $\det G\equiv 1$ (so that $\kappa(\det G)=0$ automatically) is in
\cite{BKS1} (see Theorem~23.1 there). Essentially, it was proved in
\cite{BasKS03}, with sufficiency following from Theorems~3.4, 6.1 and
necessity from Theorem~3.5 there.

Our next goal is the $APW$ factorization criterion in the not
necessarily canonical case.

\begin{thm}\label{th:K}
Let $G$ be a ${2\times 2}$ invertible $APW$ matrix function. Denote
$\delta=\kappa(\det G)$. Then $G$ admits an $APW$ factorization if
and only if  the Riemann-Hilbert problem  \eq{2.28}
e_{-\frac{\delta}{2}} G\psi_+=\psi_-\,,\; \psi_\pm\in (APW^\pm)^2 \en
admits a solution $(\psi_+,\psi_-)$ such that \eq{2.29}
\tilde\psi_+:=e_{-\tilde \delta}\psi_+\in CP^+\text{ for some } \tilde
\delta\geq 0 \text{ and }\psi_-\in CP^-.\en If this is the case, then the
partial $AP$ indices of $G$  are $\delta_1=-\tilde
\delta+\frac{\delta}{2}$, $\delta_2=\tilde \delta+\frac{\delta}{2}$ and
the factors $G_\pm$ can be chosen in such a way that $\tilde\psi_+$
is the first column of $G_+$ and $\psi_-$ is the first column of $G_-.$
\end{thm}

\begin{proof}
If $G$ admits an $APW$ factorization, then $\delta=\delta_1+\delta_2$
due to (\ref{mm}). In its turn,
$\psi_+=e_{\frac{\delta}{2}-\delta_1}g_1^+$\,,\, $\psi_-=g_1^-$  is a
solution of  (\ref{2.28})  if $\frac{\delta}{2}-\delta_1\ge 0$. It remains to
set  $\tilde \delta=\frac{\delta}{2}-\delta_1$ in order to satisfy
(\ref{2.29}) by analogy with Theorem~\ref{th:J}. Formulas
$\delta_1=\frac{\delta}{2}-\tilde\delta$,
$\delta_2=\frac{\delta}{2}+\tilde\delta$  for the partial $AP$ indices
then also hold.

Suppose now that  (\ref{2.28}) has a solution for which (\ref{2.29})
holds.  From the corona theorem in the $APW$ setting (see
\cite[Chapter 12]{BKS1}),  there exist $h_\pm=(h_{1\pm},h_{2\pm})\in
(APW^\pm)^2$ such that \eq{2.30}
\psi_{1-}h_{1-}+\psi_{2-}h_{2-}=1\,,\;\; e_{-\tilde
\delta}(\psi_{1+}h_{1+}+\psi_{2+}h_{2+})=1. \en In other words,
the matrix functions \eq{2.31} H_+=\left[%
\begin{array}{cc}
  e_{-\tilde\delta}\psi_{1+} & -h_{2+} \\
e_{-\tilde\delta}\psi_{2+} & h_{1+} \\
\end{array}%
\right]\;\;\; \text{and}\;\;\; H_-=\left[%
\begin{array}{cc}
  \psi_{1-} & -h_{2-} \\
\psi_{2-} & h_{1-} \\
\end{array}%
\right]\en have determinants equal to $1$ and are therefore
invertible in $(APW^+)^{2\times 2}$ and $(APW^-)^{2\times 2}$
respectively. Thus the matrix functions $G_1=H_-^{-1}GH_+$ and $G$
are only simultaneously $APW$ factorable, and their partial $AP$
indices coincide.

For the first column of $G_1$, taking (\ref{2.30}) into account, we have
\[
e_{-\tilde\delta}H_-^{-1}G\psi_+=e_{\frac{\delta}{2}-\tilde\delta}H_-^{-1}\psi_-=\left[%
\begin{array}{c}
  e_{\frac{\delta}{2}-\tilde\delta}  \\
0  \\
\end{array}%
\right]. \] Thus the second diagonal entry in $G_1$ must be equal
to
$$e_{\tilde\delta-\frac{\delta}{2}}\det G=\gamma_-e_{\tilde\delta+\frac{\delta}{2}}\gamma_+^{-1},$$
where \[ \det G = \gamma_- e_\delta \gamma_+^{-1} \] is a
factorization of the scalar $APW$ function $\det G$. Consequently,
\eq{2.33} G_1=\left[%
\begin{array}{cc}
 1  & 0 \\
  0 & \gamma_- \\
\end{array}%
\right] \left[%
\begin{array}{cc}
e_{\frac{\delta}{2}-\tilde\delta}  & g \\
  0 & e_{\frac{\delta}{2}+\tilde\delta}\\
\end{array}%
\right]  \left[%
\begin{array}{cc}
1 & 0 \\
  0 & \gamma_+\\
\end{array}%
\right]^{-1}\en with $g\in APW$ given by $g=[1\; \;0]G_1 [0\;\;
\gamma_+]^T$. Finally, the middle factor in the right-hand side of
(\ref{2.33}) is $APW$ factorable with the partial indices
$\frac{\delta}{2}-\tilde\delta$, $\frac{\delta}{2}+\tilde\delta$ equal to
the mean motions of its diagonal entries:
\eq{2.34}\left[%
\begin{array}{cc}
e_{\frac{\delta}{2}-\tilde\delta}  & g \\
  0 & e_{\frac{\delta}{2}+\tilde\delta}\\
\end{array}%
\right] =\left[%
\begin{array}{cc}
 1  & g_-\\
  0 & 1 \\
\end{array}%
\right] \left[%
\begin{array}{cc}
e_{\frac{\delta}{2}-\tilde\delta}  & 0 \\
  0 & e_{\frac{\delta}{2}+\tilde\delta}\\
\end{array}%
\right] \left[%
\begin{array}{cc}
 1  & -g_+\\
  0 & 1 \\
\end{array}%
\right]^{-1}.\en The only condition on  $g_\pm\in APW^\pm$ is
\eq{2.35} g
e_{\tilde\delta-\frac{\delta}{2}}=g_++g_-e_{2\tilde\delta},\en and it can
be satisfied  since $\tilde\delta\geq 0$.  Clearly, making use of
(\ref{descan}) we can always choose $G_\pm$ in such a way that
$\tilde\psi_+$ is the first column of $G_+$ and $\psi_-$ is the first
column of $G_-$.
\end{proof}
The proof of the preceding theorem provides, via (\ref{2.31}),
(\ref{2.33})--(\ref{2.35}), formulas for an $APW$ factorization of
$G=H_-^{-1}G_1H_+$, in terms of the solutions to (\ref{2.28}) and the
corona problems  (\ref{2.30}).


\section{Applications to a class of matrices with a spectral gap near zero}\label{applications}

We consider now the factorability problem for a class of
triangular matrix functions, closely related to the study of
convolution equations on an interval of finite length $\lambda$
(see, e.g., \cite[Section 1.7]{BKS1} and references therein), of the form \eq{4.1} G=\left[%
\begin{array}{cc}
  e_{-\lambda} & 0 \\
 g & e_\lambda \\
\end{array}%
\right] . \en Throughout this section we assume that
\begin{equation}\label{S5.2}
    g=a_-e_{-\beta} + a_+e_\nu \text{ for some } a_\pm\in H_\infty^\pm \text{ and } 0\leq\nu,
    \beta\leq\lambda,\ \nu+\beta>0.
\end{equation}

Representation (\ref{S5.2}), when it exists, is not unique. In particular,
it can be rewritten as \[ g=\tilde a_- e_{-\tilde\beta}+\tilde
a_+e_{\tilde\nu} \]  with  \eq{varrep} \tilde\nu\in [0,\nu], \
\tilde\beta\in [0,\beta],\  \tilde a_+=a_+e_{\nu-\tilde\nu},\  \tilde
a_-=a_-e_{\tilde\beta-\beta}.\en

Among all the representations (\ref{S5.2}) choose those with the
smallest  possible value of \eq{Ng} N=
\left\lceil\frac{\lambda}{\nu+\beta}\right\rceil,  \en where as usual
$\lceil x \rceil$ denotes the smallest integer which is greater or equal
to $x\in\R$. Of course, $N\geq 1$ due to the positivity of
$\frac{\lambda}{\nu+\beta}$.

Formula (\ref{Ng}) means that \[ N-1< \frac{\lambda}{\nu+\beta}\leq
N.\]  Decreasing $\beta,\nu$ as described in (\ref{varrep}), we may
turn the last inequality into an equality. In other words, without loss of
generality we may (and will) suppose that \eq{S5.4}
\frac{\lambda}{\nu+\beta}= N \en is an integer.

We remark that even under condition (\ref{S5.4}) representation
(\ref{S5.2}) may not be defined uniquely.

 Given $N\geq 1$, we denote by $S_{\lambda,N}$ the
class of functions $g$ satisfying (\ref{S5.2}),  (\ref{S5.4}) for which
\begin{equation}\label{S5.15}
  b_+:=  e_{\frac{\beta}{N-1}}a_-\in H_\infty^+,\quad  b_-:= e_{-\frac{\nu}{N-1}}a_+\in H_\infty^-\text{ if }N>1.
\end{equation}
By ${\mathfrak{S}}_{\lambda,N}$ we denote the class of $2\times 2$
matrix functions $G$ of the form (\ref{4.1}) with $g\in S_{\lambda,N}$.
\begin{rmk}\label{rmk:S5.1}
If $g\in S_{{\lambda,N}}$ with $N>1$, then necessarily in {\em
(\ref{S5.2})} $\beta,\nu>0$. Indeed, if say $\nu=0$, then {\em
(\ref{S5.15})} implies that $a_+$ is a constant. Consequently, $g\in
H_\infty^-$, and setting $a_-=g$, $a_+=0$, $\beta=0$, $\nu=\lambda$
in {\em (\ref{S5.2})} would yield $N=1$
--- a contradiction with our convention to choose the smallest
possible value of $N$. Note also that, due to {\em (\ref{S5.15})},
$a_\pm$ are entire functions when $N>1$.
\end{rmk}
 We start by determining a solution to (\ref{RH}) for $G$ in
${{\mathfrak{S}}}_{\lambda,N}$.

\begin{thm}\label{th:S5.4}
Let $G\in {{\mathfrak{S}}}_{{\lambda,N}}$, with $g$ given by {\em
(\ref{S5.2})}.  Then
\begin{eqnarray}
\phi_{1+} = e_{\lambda-\nu} \sum_{j=0}^{N-1}\left((-1)^j a_+^{N-1-j}a_-^j\,
e_{-j\frac{\lambda}{N}}\right)\hspace{0.2cm}&,& \phi_{2+}= -a_+^N\,,\label{S5.16}\\
\phi_{1-} = e_{-\lambda}\phi_{1+}\hspace{4.2cm} &,&  \phi_{2-}=(-1)^{N-1}a_-^N\, \hspace{0.5cm} \label{S5.17}
\end{eqnarray} deliver a solution $\phi_\pm=(\phi_{1\pm},\phi_{2\pm})$ to the Riemann-Hilbert problem {\em (\ref{RH})}.
\end{thm}

\begin{proof}
A direct computation based on the equality \[
x^N+(-1)^{N-1}y^N=(x+y)\sum_{j=0}^{N-1}\left((-1)^j
x^{N-1-j}\,y^j\right)
\] shows that $G\phi_+=\phi_-$.  Obviously, $\phi_{2\pm}\in
H_\infty^\pm$. So, it remains to prove only that $\phi_{1\pm}\in
H_\infty^\pm$. For $N=1$, this is true because the definition of
$\phi_{1+}$ from (\ref{S5.16}) collapses to $\phi_{1+}=e_\beta$. The
case $N>1$ is slightly more involved.

Namely, for $N>1$ from (\ref{S5.15}) it follows that
$$e_{\frac{\beta}{N-1}}a_-=b_+ \in H_\infty^+,$$
so that \eq{4.9} \phi_{1+} = \sum_{j=0}^{N-1}\left((-1)^j
a_+^{N-1-j}\,b_+^j \,e_{\beta-j\frac{\beta}{N-1}}\,e_{\lambda-(j+1)
\frac{\lambda}{N}}\right)\in H_\infty^+\,.\en Analogously, from
$$e_{-\frac{\nu}{N-1}}a_+=b_-\in H_\infty^-$$
we have \eq{Y2} \phi_{1-} = \sum_{j=0}^{N-1}\left((-1)^j \,b_-^{N-1-j}\,
{a}_-^j \,e_{-j(\frac{\nu}{N-1}+\frac{\lambda}{N})}\right)\in
H_\infty^-\,.\en
\end{proof}

This theorem, along with Theorem \ref{th:FA}, allows to establish
sufficient conditions, which in some cases are also necessary, for
invertibility in $(H_p^+)^2$, $p>1$, of Toeplitz operators with symbol
$G\in {{\mathfrak{S}}}_{{\lambda,N}}$. To invoke Theorem \ref{th:FA},
however, we need to be able to check when the pairs $(\phi_{1\pm},
\phi_{2\pm})$ defined by (\ref{S5.16}), (\ref{S5.17}) belong to $CP^+$
or $CP^-$. The following result from \cite{CaDi08} (see Theorem~2.3
there) will simplify this task.
\begin{thm}\label{th:CD} Let a $2\times 2$ matrix function $G$ and its inverse $G^{-1}$
be analytic and bounded in a strip
\begin{equation}\label{S.5.17strip}
S=\{\xi\in\C : -\varepsilon_2<\Im \xi<\varepsilon_1\}\;\;\;
\text{with}\;\;\; \varepsilon_1, \varepsilon_2\in [0,+\infty[\,,
\end{equation}
and let $\phi_\pm\in (H_\infty^\pm)^2$ satisfy {\em (\ref{RH})}. Then
$\phi_+\in CP^+$ (resp. $\phi_-\in CP^-$) if and only if
\begin{equation}\label{S.5.17B}
\inf_{\C^++i\varepsilon_1}
(|\phi_{1+}|+|\phi_{2+}|)>0\hspace{0.5cm}\left(\text{resp.,
\;}\inf_{\C^--i\varepsilon_2} (|\phi_{1-}|+|\phi_{2-}|)>0\right)
\end{equation}
and one of the following (equivalent) conditions is satisfied:
\begin{eqnarray}
  \inf_S (|\phi_{1+}|+|\phi_{2+}|) & > & 0, \label{S.5.17C}\\
  \inf_S (|\phi_{1-}|+|\phi_{2-}|) & > & 0.\label{S.5.17D}
\end{eqnarray}
\end{thm}

Here and in what follows, we identify the functions $\phi_{1+},
\phi_{2+}$ (resp., $\phi_{1-}, \phi_{2-}$) with their analytic extensions
to $\C^+-i\varepsilon_2$ (resp. $\C^++i\varepsilon_1$) and, for any
real-valued function $\phi$ defined on $S$, abbreviate $\inf_{\zeta\in
S}\phi(\zeta)$ to $\inf_S\phi$.

We will see that for $G\in {{\mathfrak{S}}}_{{\lambda,N}}$, $N\geq 1$,
the behavior of the solutions ``at infinity", that is, condition
(\ref{S.5.17B}) for sufficiently big  $\varepsilon_1,\varepsilon_2>0$, is
not difficult to study.   Therefore, due to Theorem~\ref{th:CD}, we will
be left with studying the behavior of $\phi_+$ or $\phi_-$ in a strip of
the complex plane. According to the next result this, in turn, can be
done in term of the functions $a_\pm$ from (\ref{S5.2}) or,
equivalently, of $g_\pm$ defined by
\[
g_+=e_\nu a_+\,, \;\;\; g_-=e_{\nu-\frac{\lambda}{N}}a_-
\]
It should be noted that, for $N>1$, $a_\pm$ and $g_\pm$ are entire
functions. Moreover, even if the behaviour of $a_+$ and $a_-$ in a strip
$S$ may be difficult to study, it is clear from (\ref{S5.16}) and
(\ref{S5.17}) that this is in general a much simpler task than that of
checking whether (\ref{S.5.17B}) is satisfied using the expressions for
$\phi_{1\pm}, \phi_{2\pm}$.
\begin{lem}\label{th:TCD}
Let $G\in {{\mathfrak{S}}}_{{\lambda,N}}$ for some $N>1$, and let
$\phi_\pm$ be given by {\em (\ref{S5.16}), (\ref{S5.17})}. Then for any
strip {\em (\ref{S.5.17strip})} we have
\begin{equation}\label{S.5.17FA}
\inf_S (|\phi_{1+}|+|\phi_{2+}|) > 0 \Longleftrightarrow \inf_S
(|a_{+}|+|a_{-}|) > 0 \Longleftrightarrow \inf_S (|g_{+}|+|g_{-}|) >
0.
\end{equation}
\end{lem}

\begin{proof}
Since the last two conditions in (\ref{S.5.17FA}) are obviously
equivalent, and (\ref{S.5.17C}) is equivalent to (\ref{S.5.17D}) due to
Theorem~\ref{th:CD}, we need to prove only that
$$\inf_S (|\phi_{1+}|+|\phi_{2+}|) > 0 \Longleftrightarrow
\inf_S (|a_{+}|+|a_{-}|) > 0 .$$

Suppose first that
$$\inf_{\xi\in S}(\,|a_+(\xi)|+|a_-(\xi)|\,)=0.$$

Then there is a sequence $\{\xi_n\}_{n\in \N}$ with $\xi_n\in S$ such
that $a_+(\xi_n)\rightarrow 0$ \; and \;$a_-(\xi_n)\rightarrow 0$.
Taking into account the expressions for $\phi_{1+}, \phi_{2+}$ given
by (\ref{S5.16}), we must have $\phi_{1+}(\xi_n)\rightarrow 0$ \; and
\;$\phi_{2+}(\xi_n)\rightarrow 0$. Therefore,
$$\inf_{\xi\in S}(\,|\phi_{1+}(\xi)|+|\phi_{2+}(\xi)|\,)=0.$$

Conversely, if  $$\inf_{\xi\in S}(\,|\phi_{1+}(\xi)|+|\phi_{2+}(\xi)|\,)=0,$$
then for some sequence $\{\xi_n\}$ with $\xi_n\in S$ for all $n\in \N$,
we have $ \phi_{1+}(\xi_n) \rightarrow  0 \hspace{0.2cm} \text{and}
\hspace{0.2cm}\phi_{2+}(\xi_n) \rightarrow  0\,.$ Thus, from the
expression for $\phi_{2+}$ given by (\ref{S5.16}), it follows that $
a_{+}(\xi_n) \rightarrow 0$. From the expression for $\phi_{1+}$ in
(\ref{S5.16}), we then conclude
$$a_-^{N-1}=(-1)^{N-1}e_{\nu-\frac{\lambda}{N}}\phi_{1+}+(-1)^{N}e_{\frac{N-1}{N}\lambda}
\sum_{j=0}^{N-2}\left((-1)^j a_+^{N-1-j}a_-^j\,
e_{-j\frac{\lambda}{N}}\right).$$ Since $ \phi_{1+}(\xi_n) \rightarrow
0\hspace{0.2cm} \text{and} \hspace{0.2cm} a_{+}(\xi_n) \rightarrow
0$, then also $a_{-}(\xi_n) \rightarrow 0\,$ and therefore
$$\inf_{\xi\in S}(\,|a_+(\xi)|+|a_-(\xi)|\,)=0.$$
\end{proof}

We can now state the following.
 \begin{thm}\label{th:Inv} Let $G\in {{\mathfrak{S}}}_{{\lambda,N}}$ for some
$N\in\N$, and let $\phi_\pm$ be the solutions to {\em (\ref{RH})}
 given by {\em (\ref{S5.16}), (\ref{S5.17})}. Then:
 \begin{description}
    \item[(i)] For $N=1$, $\phi_\pm\in CP^\pm$ if and only if
\begin{equation}\label{S.5.17G}
\inf_{\C^++i\varepsilon_1}|a_+|>0\,,\;\;\inf_{\C^--i\varepsilon_2}|a_-|>0 \text{ for some } \varepsilon_1,\varepsilon_2>0.
\end{equation}
    \item[(ii)] For $N>1$, $\phi_\pm\in CP^\pm$ if and only if, with
        $b_+, b_-$ defined by {\em (\ref{S5.15})},
    \begin{equation}\label{S.5.17H}
\inf_{\C^++i\varepsilon_1}(|b_+|+|a_+|)>0\,,\;\inf_{\C^--i\varepsilon_2}(|b_-|+|a_-|)>0 \text{ for some }  \varepsilon_1,\varepsilon_2>0
\end{equation}
and, for any $S$ of the form {\em (\ref{S.5.17strip})},
  \begin{equation}\label{S.5.17I}
\inf_{S}(|a_+|+|a_-|)>0.
\end{equation}
 \end{description}
\end{thm}

\begin{proof}Part {\bf (i)} follows immediately from the explicit formulas
\begin{equation}\label{S5.5}
\phi_+=(e_{\beta}\,,\, -a_+)\,, \hspace{0.8cm}
\phi_-=(e_{-\nu}\,,\, a_-). \,
\end{equation}.

{\bf (ii)} For $N>1$ we have, from (\ref{S5.16})--(\ref{Y2}),
\begin{eqnarray}
\phi_{1+}  &=& (-1)^{N-1}b_+^{N-1}+\sum_{j=0}^{N-2} \left((-1)^j a_+^{N-1-j}b_+^j
e_{(N-1-j)\frac{\lambda-\nu}{N-1}}\right), \hspace{0.5cm} \phi_{2+}=-a_+^N\,,\label{S.5.17J}\\
\phi_{1-} &=& b_-^{N-1}+\sum_{j=1}^{N-1}\left((-1)^ja_-^jb_-^{N-1-j}e_{-j\frac{\lambda-\beta}{N-1}}\right),
\hspace{0.5cm} \phi_{2-}=(-1)^{N-1}a_-^N. \label{S.5.17K}
\end{eqnarray}
Since $\nu,\beta<\lambda$ when $N>1$, we see that for any
sequence $\{\xi_n\}$ with $\xi_n\in\C^+$ and $\Im (\xi_n)\rightarrow
+\infty$,
\begin{equation}\label{S.5.17L}
    \left|\phi_{1+}-(-1)^{N-1}b_+^{N-1}\right|_{(\xi_n)}\rightarrow 0\,,
\end{equation}
and, for any sequence $\{\xi_n\}$ with $\xi_n\in\C^-$ and $\Im
(\xi_n)\rightarrow -\infty$,
\begin{equation}\label{S.5.17M}
    \left|\phi_{1-}-b_-^{N-1}\right|_{(\xi_n)}\rightarrow 0\,.
\end{equation}

It follows from (\ref{S.5.17J})--(\ref{S.5.17M}) that there exist
$\varepsilon_1,\varepsilon_2>0$ such that (\ref{S.5.17H}) holds if and
only if there exist $\varepsilon_1,\varepsilon_2>0$ such that
(\ref{S.5.17B}) holds. Moreover, by Lemma~\ref{th:TCD}, (\ref{S.5.17I})
is equivalent to (\ref{S.5.17C}), thus the result follows from
Theorem~\ref{th:CD}.
\end{proof}

Note that $\det G\equiv 1$ for all matrix functions of the form
(\ref{4.1}). Therefore, Theorems \ref{th:FA}, \ref{th:2.7A} and \ref{th:Inv}
combined imply the following.

\begin{cor}\label{cor:inv}
Let the assumptions of Theorem~{\em \ref{th:Inv}} hold. Then
condition {\em (\ref{S.5.17G})} (for $N=1$) and {\em (\ref{S.5.17H}),
(\ref{S.5.17I})} (for $N>1$) imply the invertibility of $T_G$. The
converse is also true (and, moreover, $G$ admits a bounded
canonical factorization) provided that $G\in {\mathcal B}^{2\times 2}$.
\end{cor}
For $N=1$, this result was proved (assuming $\lambda=1$, which
amounts to a simple change of variable) in \cite{CaSantos00},
Theorem~4.1 and Corollary~4.5.

For the particular case when $a_-$  (or $a_+$) is just a single
exponential function, condition (\ref{S.5.17I}) is always satisfied and
we can go deeper in the study of the properties of $T_G$. Before
proceeding in this direction, however, it is useful to establish a more
explicit characterization of the classes $S_{\lambda,N}$ under the
circumstances. Without loss of generality, let us concentrate on the
case when $a_-$ is an exponential.

\begin{lem}\label{l:se}Given $\lambda>0$, let \eq{C1}g=e_{-\sigma}+g_+,\en
where $g_+\in H_\infty^+$ is not identically zero, and
$0<\sigma<\lambda$. Then $g\in S_{\lambda,N}$ for some $N\in\N$
if and only if \eq{C2} e_{-\nu}g_+\in H_\infty^+,\quad
e_{-\frac{N}{N-1}\nu}g_+\in H_\infty^- \en for some \eq{C3}\nu\in
\left[\frac{\lambda}{N}-\sigma,
\frac{\lambda}{N}-\frac{N-1}{N}\sigma\right]\en (of course, the
second condition in {\em (\ref{C2})} applies only for $N>1$).
\end{lem}
Note that conditions (\ref{C2}), (\ref{C3}) imply \[
e_{-\frac{\lambda}{N}+\sigma}g_+\in H_\infty^+,\quad
e_{-\frac{\lambda}{N-1}+\sigma}g_+\in H_\infty^-,\] and therefore may
hold for at most one value of $N$.
\begin{proof}{\sl Necessity.} Suppose $g\in S_{\lambda,N}$. Comparing (\ref{S5.2}) and (\ref{C1}) we see that
\eq{C4} a_-=e_{\beta-\sigma}\in H_\infty^-,\quad  a_+=e_{-\nu}g_+\in
H_\infty^+. \en On the other hand,  (\ref{S5.15}) takes the form \eq{C5}
e_{\frac{N}{N-1}\beta-\sigma}\in H_\infty^+,\quad
e_{-\frac{N}{N-1}\nu}g_+\in H_\infty^-.\en The first containments in
(\ref{C4}), (\ref{C5}) are equivalent to \[
\frac{N-1}{N}\sigma\leq\beta\leq\sigma,
\] which along with (\ref{S5.4}) yields that
$\nu=\frac{\lambda}{N}-\beta$ satisfies (\ref{C3}). The second
containments in (\ref{C4}), (\ref{C5}) then imply (\ref{C2}).

{\sl Sufficiency.} Given (\ref{C2}), (\ref{C3}), let $
\beta=\frac{\lambda}{N}-\nu$, and define $a_\pm$ by (\ref{C4}). Then
(\ref{S5.2}), (\ref{S5.4}) and (\ref{S5.15}) hold (the latter for $N>1$).
\end{proof}

\begin{thm}\label{th:CB}Let $G$ be given by {\em (\ref{4.1})} with $g$ of the form
\eq{C18} g=e_{-\sigma}+e_\mu a_+,\quad \mu,\sigma>0, \ a_+\in
H_\infty^+,\en where $\mu+\sigma\geq\lambda$. Then the Toeplitz
operator $T_G$ is invertible if (and only if, provided that $G\in
{\mathcal B}^{2\times 2}$) \eq{C19} \mu+\sigma=\lambda\text{ and }
\inf_{\C^++i\varepsilon}\abs{a_+}>0\text{ for some } \varepsilon
>0,\en and $T_G$ is not semi-Fredholm if $\mu+\sigma>\lambda$.\end{thm}
\begin{proof}Condition (\ref{C18}) implies that $g\in S_{\lambda,1}$ with
$\beta=\sigma$, $\nu=\lambda-\sigma$, and a solution to (\ref{RH})
is given by \[ \phi_+=(e_\sigma, -e_{\mu+\sigma-\lambda}a_+), \quad
\phi_-=(e_{\sigma-\lambda}, 1).\]  Clearly, $\phi_-\in CP^-$, while
$\phi_+\in CP^+$ if and only if (\ref{C19}) holds. The part of the
statement pertinent to the case $\lambda=\sigma+\mu$ now follows
from Theorems~\ref{th:FA},~\ref{th:2.7A}.

For $\mu+\sigma>\lambda$, following the proof of \cite[Theorem
5.3]{CDR} observe that $\frac{1-e_{-\gamma}(z)}{z}\phi_\pm(z)$
deliver a solution to (\ref{RH}) in $L_p$, for any $\gamma$ between
$0$ and $\min\{\sigma,\mu+\sigma-\lambda\}$. Thus, the operator
$T_G$ has an infinite dimensional kernel in $(H_p^+)^2$ for any $p\in
(1,\infty)$.

Denote by $G^{-T}$ the transposed of $G^{-1}$. A direct computation
shows that for the matrix under consideration, due to its algebraic
structure, \eq{GT} G^{-T}= \left[\begin{matrix}0 & -1\\ 1 &
0\end{matrix}\right] G \left[\begin{matrix}0 & 1\\ -1 &
0\end{matrix}\right]. \en Therefore, the operator $T_{G^{-T}}$ also has
an infinite dimensional kernel. But this means (see, e.g., \cite[Section
3.1]{LS}) that the cokernel of $T_G$ is infinite dimensional. Therefore,
the operator $T_G$ is not even semi-Fredholm on $(H_p^+)^2$,
$1<p<\infty$. \end{proof}

\begin{thm}\label{th:CD1}Let, as in Theorem~{\em \ref{th:CB}}, {\em (\ref{4.1})} and {\em (\ref{C18})}  hold,
but now with \[ \mu\in \left[\frac{\lambda}{N}-\sigma,
\frac{\lambda}{N}-\frac{N-1}{N}\sigma\right] \text{ and }
e_{-\frac{\mu}{N-1}}a_+\in H_\infty^- \] for some integer $N>1$. Then
$T_G$ is invertible if (and only if, for $G\in {\mathcal B}^{2\times 2}$)
for some $\varepsilon>0$ one of the following three conditions holds:
\eq{C22} \sigma+\mu=\frac{\lambda}{N} \text{ and }
\inf_{\C^++i\varepsilon}\abs{a_+}>0, \en or \eq{C23}
\frac{N-1}{N}\sigma+\mu=\frac{\lambda}{N} \text{ and }
\inf_{\C^--i\varepsilon}\abs{e_{-\frac{\mu}{N-1}}a_+}>0, \en or \[
\inf_{\C^++i\varepsilon}\abs{a_+}>0, \text{ and }
\inf_{\C^--i\varepsilon}\abs{e_{-\frac{\mu}{N-1}}a_+}>0. \] If, on the
other hand, \eq{C25} \sigma+\mu>\frac{\lambda}{N} \text{ and }
e_{\delta-\frac{\mu}{N-1}}a_+\in H_\infty^- \en or \eq{C26}
\frac{N-1}{N}\sigma+\mu<\frac{\lambda}{N} \text{ and }
e_{-\delta}a_+\in H_\infty^+ \en for some $\delta>0$, then $T_G$ is not
even semi-Fredholm.
\end{thm}
\begin{proof}According to Lemma~\ref{l:se},  $G\in{\frak S}_{\lambda,N}$. Moreover,
one can choose in (\ref{S5.2}) $\nu=\mu$,
$\beta=\frac{\lambda}{N}-\mu$  and
$a_-=e_{\frac{\lambda}{N}-\mu-\sigma}$. Then formulas (\ref{S5.16}),
(\ref{S5.17}) yield the following solution to (\ref{RH}):
\begin{align*} \phi_{1+}= &
e_{\lambda-N(\mu+\sigma)+\sigma}\sum_{j=0}^{N-1}\left((-1)^ja_+^{N-1-j}e_{(N-1-j)(\mu+\sigma)}\right),\\
\phi_{2+}= & -a_+^N,\\
\phi_{1-}= &
\sum_{j=0}^{N-1}\left((-1)^j(e_{-\frac{\mu}{N-1}}a_+)^{N-1-j}e_{-j(\frac{N}{N-1}\mu+\sigma)}\right),\\
\phi_{2-}= & (-1)^{N-1}e_{\lambda-N(\sigma+\mu)}.
\end{align*} Clearly, $(\phi_{1-},\phi_{2-})\in CP^-$ if and only if the first condition in (\ref{C22}) or the second condition
in (\ref{C23}) holds. Similarly, $(\phi_{1+},\phi_{2+})\in CP^+$ is
equivalent to the first condition in (\ref{C23}) or the second condition
in (\ref{C22}). Since the first conditions in (\ref{C22}), (\ref{C23})
cannot hold simultaneously, the statement regarding the invertibility
of $T_G$ now follows from Theorems~\ref{th:FA},~\ref{th:2.7A}.

If (\ref{C25}) or (\ref{C26}) holds, then
$\phi_-=e_{-\tilde{\delta}}\tilde{\phi}_-$ or
$\phi_+=e_{\tilde{\delta}}\tilde{\phi}_+$ with $\tilde{\delta}>0$,
$\tilde{\phi}_\pm\in(H_\infty^\pm)^2$, respectively. It follows that the
kernel of  $T_G$ is infinite dimensional, as in the proof of
Theorem~\ref{th:CB}.   Using (\ref{GT}), we in the same manner derive
that the cokernel of $T_G$ also is infinite dimensional. So, $T_G$ is not
semi-Fredholm.
\end{proof}

\section{AP matrix functions with a spectral gap around zero}\label{GAP}

The results of the previous section take a particular and, in some
sense, more explicit form when considered in the almost periodic
setting. The first natural question is, which functions $g\in AP$ belong
to $S_{{\lambda,N}}$ for some $N\in\N$, with $a_\pm\in AP^\pm$ in
(\ref{S5.2}).

According to Remark~\ref{rmk:S5.1}, we may have $0\in\Omega(g)$
only if $N=1$ and, in addition, $g=a_-+a_+e_\lambda$ with
$0\in\Omega(a_-)$ or $g=a_-e_{-\lambda}+a_+$ with
$0\in\Omega(a_+)$. In either case the operator $T_G$ is invertible, as
can be deduced from the so called one sided case, see \cite[Section
14.1]{BKS1}. The easiest way to see that directly, however, is by
observing that problem (\ref{RH}) has a solution on $CP^\pm$:
$\phi_+=(1,-a_+)$, $\phi_-=(e_{-\lambda}, a_-)$ in the first case,
$\phi_+=(e_\lambda,-a_+)$, $\phi_-=(1, a_-)$ in the second.

Therefore, in what follows we restrict ourselves to the case
$0\notin\Omega(g)$. Then
\begin{equation}\label{S.5.17A}
g=g_-+g_+\text{ with } g_\pm\in AP^\pm,\ 0\notin\Omega(g_\pm)
\end{equation} with $g_\pm$ uniquely defined by $g$. Comparing with (\ref{S5.2}), we have
\eq{5.1A} g_+=a_+e_\nu,\quad g_-=a_-e_{-\beta}.\en

Let
\begin{eqnarray}
\eta_{1-}=-\sup\Omega(g_-) & , & \eta_{2-}=-\inf\Omega(g_-), \label{S5.18}\\
 \eta_{1+}=\inf \Omega(g_+) & , & \eta_{2+}=\sup \Omega(g_+).\label{S5.19}
\end{eqnarray}
Here $\Omega(g_+), -\Omega(g_-)$ are thought of as subsets of
$\R_+$ (possibly empty), so that $\eta_{1\pm},\eta_{2\pm}\in
[0,+\infty]\cup\{-\infty\}$.

\begin{thm}\label{th:S5.5}Let $g$ be given by {\em (\ref{S.5.17A})}. Then \begin{itemize}
\item[(i)] $g\in S_{\lambda,1}$ if and only if
    $\eta_{1+}+\eta_{1-}\geq\lambda$; \item[(ii)] $g\in
    S_{\lambda,N}$ with $N>1$ if and only if \end{itemize} \eq{S5.20}
    N=\left\lceil\frac{\lambda}{\eta_{1-}+\eta_{1+}}\right\rceil ,\en while
    \eq{S5.21} \eta_{1-}\geq \frac{N-1}{N}\eta_{2-}, \  \eta_{1+}\geq
    \frac{N-1}{N}\eta_{2+},\ \eta_{2+}+\eta_{2-}\leq\frac{\lambda}{N-1}.
    \en Under these conditions, any $\nu$ satisfying \eq{S5.25}
    M:=\max\left\{\frac{\lambda}{N}-\eta_{1-},\,
    \frac{N-1}{N}\eta_{2+}\right\}\leq\nu\leq
    \min\left\{\eta_{1+},\frac{\lambda-(N-1)\eta_{2-}}{N}\right\}=:m\en
    and \eq{S5.24} a_+=g_+e_{-\nu},\ \beta=\frac{\lambda}{N}-\nu, \
    a_-=g_-e_\beta \en deliver a representation {\em (\ref{S5.2})}.
\end{thm}
\begin{proof} (i) If $g\in S_{\lambda,1}$, then from (\ref{5.1A}) with $\nu+\beta=\lambda$ it
follows that $\eta_{1+}+\eta_{1-}\geq\lambda$. Conversely, setting
$a_\pm=0$ if $g_\pm=0$, $a_+=g_+e_{-\eta_{1+}}$,
$a_-=g_-e_{\lambda-\eta_{1+}}$ if $g_+\neq 0$, and
$a_+=g_+e_{-\lambda+\eta_{1-}}$, $a_-=g_-e_{\eta_{1-}}$ if $g_-\neq
0$, we can write $g$ as in (\ref{S5.2}) with $\nu+\beta=\lambda$, so
that $g\in S_{\lambda,1}$.

(ii) {\sl Necessity.} Formulas for $a_\pm$ in (\ref{S5.24}) follow from
the uniqueness of $g_\pm$ in the representation (\ref{S.5.17A}). The
condition $a_\pm\in H_\infty^\pm$ is therefore equivalent to \eq{c1}
\beta\leq \eta_{1-},\quad \nu\leq \eta_{1+}.\en Conditions
(\ref{S5.15}), in their turn, are equivalent to \eq{c2} \beta\geq
\frac{N-1}{N}\eta_{2-},\quad \nu\geq\frac{N-1}{N}\eta_{2+}.\en
Comparing the respective inequalities in (\ref{c1}) and (\ref{c2})
shows the necessity of the first two conditions in (\ref{S5.21}). To
obtain the third condition there, just add the two inequalities in
(\ref{c2}): \[ \beta+\nu\geq\frac{N-1}{N}(\eta_{2+}+\eta_{2-}),\]  and
compare the result with (\ref{S5.4}).

On the other hand, adding the inequalities in (\ref{c1}) yields, once
again with the use of (\ref{S5.4}), \[ \frac{\lambda}{N}=\beta+\nu\leq
\eta_{1+}+\eta_{1-}. \] So, \eq{in}
\frac{\lambda}{\eta_{1+}+\eta_{1-}}\leq N\leq
1+\frac{\lambda}{\eta_{2+}+\eta_{2-}}.\en If at least one of the
inequalities $\eta_{2\pm}>\eta_{1\pm}$ holds, the difference between
the right- and left-hand sides of the inequalities (\ref{in}) is strictly less
than 1, and therefore an integer $N$ is defined by (\ref{in}) uniquely, in
accordance with (\ref{S5.20}). Otherwise, $\eta_{1\pm}=\eta_{2\pm}$,
which means that $g=c_1e_{\eta_{1-}}+c_2e_{\eta_{1+}}$ with
$c_1,c_2\in\C\setminus\{0\}$. Since by definition $N$ is the smallest
possible number satisfying (\ref{Ng}) with $\nu,\beta$ such that
(\ref{S5.2}) holds, we arrive  again at (\ref{S5.20}).

{\sl Sufficiency.} Let (\ref{S5.21}) hold for $N$ defined by (\ref{S5.20}).
Then $m,M$ defined in (\ref{S5.25}) satisfy $M\leq m$, so that  $\nu$
may indeed be chosen as in (\ref{S5.25}). With such $\nu$, and
$a_\pm$ defined by (\ref{S5.24}), we have (\ref{S5.2}), (\ref{S5.4}), and
(\ref{S5.15}). \end{proof}

The results of Theorem \ref{th:Inv} and Corollary \ref{cor:inv},
combined with Theorem \ref{th:S5.5},  yield the following.

\begin{thm}\label{th:S5.5A}
Let $g\in S_{\lambda,N}$ be written as  {\em (\ref{S.5.17A})}, and let
$\eta_{j\pm}$ $(j=1,2)$ be defined by {\em (\ref{S5.18})--(\ref{S5.19})}.
Then the Toeplitz operator $T_G$ with symbol $G$ given by {\em
(\ref{4.1})} is invertible if (and, for $g\in APW$, only if) one of the
following conditions holds:
\begin{description}
    \item[(i)] $N=1$ and
    \begin{equation}\label{S5.30A}
    \eta_{1+}\in\Omega
    (g_+)\,,\quad -\eta_{1-}\in\Omega (g_-)\,,\quad\eta_{1+}+\eta_{1-}=\lambda\,;
\end{equation}
    \item[(ii)] $N>1$ and
    \begin{equation}\label{S5.30B}
    \eta_{1+}\in\Omega (g_+)\,,\quad -\eta_{1-}\in\Omega
    (g_-)\,,\quad
    \eta_{1+}+\eta_{1-}=\frac{\lambda}{N}\,;
\end{equation}
    \item[(iii)]$N>1$ and
    \begin{equation}\label{S5.30D}
    \eta_{1+},\, \eta_{2+}\in\Omega
    (g_+)\,,\quad\eta_{2+}=\frac{N}{N-1}\,\eta_{1+}\,;
\end{equation}
    \item[(iv)]$N>1$ and
    \begin{equation}\label{S5.30C}
    -\eta_{1-},-\eta_{2-}\in\Omega
    (g_-)\,,\quad
    \eta_{2-}=\frac{N}{N-1}\,\eta_{1-}\,;
\end{equation}
    \item[(v)] $N>1$ and
  \begin{equation}\label{S5.30E}
    \eta_{2+}\in\Omega
    (g_+)\,,\quad-\eta_{2-}\in\Omega
    (g_-)\,,\quad
     \eta_{2+}+\eta_{2-}=\frac{\lambda}{N-1}\,;
\end{equation}
\end{description}
and, whenever $N>1$,
\begin{equation}\label{S5.30J}
\inf_S (|g_+|+|g_-|)>0\text{ for any strip $S$ of the form
{\em (\ref{S.5.17strip})}}.
\end{equation}
\end{thm}

\begin{proof}
For $N=1$, (\ref{S5.30A}) is equivalent to (\ref{S.5.17G}).

For $N>1$, setting
\begin{equation}\label{x9}
a_-=e_{\frac{\lambda}{N}-\nu}g_- \;\;\text{and}\;\; a_+=e_{-\nu}g_+
\end{equation}
where $\beta=\frac{\lambda}{N}-\nu$, we deduce from (\ref{S5.15})
that
\begin{equation}\label{x10}
b_-=e_{-\frac{N\nu}{N-1}}\,g_+ \;\;\text{and}\;\; b_+=e_{\frac{\lambda-N\nu}{N-1}}\,g_-.
\end{equation}
Hence
\begin{align*}
M(a_+)\ne 0 &\quad\text{if and only if}\quad \eta_{1+}=\nu\in\Omega(g_+),\\
M(b_+)\ne 0 &\quad\text{if and only if}\quad -\eta_{2-}=-\frac{\lambda-N\nu}{N-1}\in\Omega(g_-),\\
M(a_-)\ne 0 &\quad\text{if and only if}\quad -\eta_{1-}=\nu-\frac{\lambda}{N}\in\Omega(g_-),\\
M(b_-)\ne 0 &\quad\text{if and only if}\quad \eta_{2+}=-\frac{N\nu}{N-1}\in\Omega(g_+).
\end{align*}
Thus, the first inequality in (\ref{S.5.17H}) holds if and only if either
$\eta_{1+}=\nu\in\Omega(g_+)$ or
$-\eta_{2-}=-\frac{\lambda-N\nu}{N-1}\in\Omega(g_-)$, and the
second inequality in (\ref{S.5.17H}) holds if and only if either
$-\eta_{1-}=\nu-\frac{\lambda}{N}\in\Omega(g_-)$ or
$\eta_{2+}=-\frac{N\nu}{N-1}\in\Omega(g_+)$.

Taking now $\eta_{1+}=\nu\in\Omega(g_+)$ and
$-\eta_{1-}=\nu-\frac{\lambda}{N}\in\Omega(g_-)$, we get the
equivalence of (\ref{S.5.17H}) and (\ref{S5.30B}); taking
$\eta_{1+}=\nu\in\Omega(g_+)$ and
$\eta_{2+}=-\frac{N\nu}{N-1}\in\Omega(g_+)$, we get the equivalence
of (\ref{S.5.17H}) and (\ref{S5.30D}); taking
$-\eta_{2-}=-\frac{\lambda-N\nu}{N-1}\in\Omega(g_-)$ and
$-\eta_{1-}=\nu-\frac{\lambda}{N}\in\Omega(g_-)$, we get the
equivalence of (\ref{S.5.17H}) and (\ref{S5.30C}); taking
$-\eta_{2-}=-\frac{\lambda-N\nu}{N-1}\in\Omega(g_-)$ and
$\eta_{2+}=-\frac{N\nu}{N-1}\in\Omega(g_+)$, we get the equivalence
of (\ref{S.5.17H}) and (\ref{S5.30E}). Thus, we see that (\ref{S.5.17H}) is
equivalent to one of the conditions (ii)--(v) of the theorem being
satisfied.

The result now follows from Theorem \ref{th:Inv} and Corollary
\ref{cor:inv} and the second equivalence in (\ref{S.5.17FA}).
\end{proof}

From (\ref{S5.25}) it follows that in the case (ii) we have
$\nu=\frac{\lambda}{N}-\eta_{1-}=\eta_{1+}$ and therefore
\[ \lambda\ge\max\{ N\eta_{1-}+(N-1)\eta_{2+},
N\eta_{1+}+(N-1)\eta_{2-}\},\]  in the case (iii) we have
$\nu=\frac{N-1}{N}\,\eta_{2+}=\eta_{1+}$ so that
\[ N\eta_{1+}+(N-1)\,\eta_{2-}\le\lambda\le N\eta_{1-}+(N-1)\eta_{2+},\]
in the case (iv) we have
$\nu=\frac{\lambda}{N}-\eta_{1-}=\frac{\lambda}{N}
-\frac{N-1}{N}\,\eta_{2-}$ and therefore \[ N\eta_{1-}+(N-1)\eta_{2+}\le
\lambda\le N\eta_{1+}+(N-1)\eta_{2-},\]  in the case (v) we have
$\nu=\frac{N-1}{N}\,\eta_{2+}=\frac{\lambda}{N}-\frac{N-1}{N}\,\eta_{2-}$
so that \[ \lambda\le \min\{ N\eta_{1-}+(N-1)\eta_{2+},
N\eta_{1+}+(N-1)\eta_{2-}\}.\]

We also note that if $\lambda=N\eta_{1-}+(N-1)\eta_{2+}$, then
condition (\ref{S5.30D}) is equivalent to
\begin{equation}\label{S5.30H}
\eta_{1+},\, \eta_{2+}\in\Omega
(g_+)\,,\quad\eta_{1+}+\eta_{1-}=\frac{\lambda}{N}\,;
\end{equation}
while condition (\ref{S5.30C}) is equivalent to
\begin{equation}\label{S5.30F}
-\eta_{1-},-\eta_{2-}\in\Omega(g_-)\,,\quad
\eta_{2+}+\eta_{2-}=\frac{\lambda}{N-1}\,.
\end{equation}
If $\lambda=N\eta_{1+}+(N-1)\eta_{2-}$, then condition (\ref{S5.30D})
is equivalent to (\ref{S5.30F}), while condition (\ref{S5.30C}) is
equivalent to (\ref{S5.30H}).

Observe that necessity of conditions (\ref{S5.30A})--(\ref{S5.30E})
persists for $g\in AP$ without an additional restriction $g\in APW$. To
see that, suppose that $T_G$ is invertible in one of the cases  (i)--(v)
while the respective condition (\ref{S5.30A})--(\ref{S5.30E}) fails.
Approximate $g$ by a function in $APW$ with the same $\eta_{j\pm}$
and so close to $g$ in the uniform norm that the respective Toeplitz
operator is still invertible. This contradicts the necessity of
(\ref{S5.30A})--(\ref{S5.30E}) in the $APW$ case.

It is not clear, however, whether the condition (\ref{S5.30J}) remains
necessary in the $AP$ setting.

\begin{rmk}\label{rmk:S5.5B}
Part {(i)} of Theorem~{\em\ref{th:S5.5A}} means that, for $T_G$ to be
invertible in the case when the length of the spectral gap of $g$
around zero is at least $\lambda$, it in fact must equal $\lambda$ and,
moreover, both endpoints of the spectral gap must belong to
$\Omega(g)$. In contrast with this, for $N>1$ according to parts
{(ii)--(v)} $T_G$ can be invertible when one (or both) of the endpoints
of the spectral gap around zero is missing from $\Omega(g)$, and the
length of this spectral gap can be greater than ${\lambda}/{N}$.
\end{rmk}

For $g\in APW$ Theorem~\ref{th:S5.5A} delivers the invertibility
criterion of $T_G$, and thus a necessary and sufficient condition for
$G$ to admit a  canonical $APW$ factorization. Using
Theorem~\ref{th:K}, however, will allow us to tackle the non-canonical
$AP$ factorability of $G$ as well.

{We assume from now on that $g\in APW$ is given by (\ref{S.5.17A}),
so that in fact $g_\pm\in APW^\pm$, and that $g\in S_{\lambda,N}$
as described by Theorem~\ref{th:S5.5}.}

In the notation of this theorem,  for $N=1$ we have
$\eta_{1+}+\eta_{1-}\geq \lambda$ --- the so called {\em big gap}
case, --- and a solution to (\ref{RH}) is given by
\begin{eqnarray}
  \phi_+ &=& (e_{\lambda-\nu}, -e_{-\nu+\eta_{1+}}\tilde g_+),\label{S5.31} \\
 \phi_- &=& (e_{-\nu},e_{\lambda-\nu-\eta_{1-}}\tilde g_-),\label{S5.32}
\end{eqnarray}
where
\begin{eqnarray}
 \tilde g_+&=& e_{-\eta_{1+}}g_+\;\;\;\;\; (0=\inf \Omega(\tilde g_+)),\label{S5.33} \\
\tilde g_- &=& e_{\eta_{1-}}g_- \;\;\;\;\; (0=\sup \Omega(\tilde
g_-)), \label{S5.34}
\end{eqnarray}
\begin{equation}
\max\{0,\lambda-\eta_{1-}\}\le\nu\le\min\{\eta_{1+},\lambda\}.\label{S5.35}
\end{equation}

Knowing these  solutions and using Theorem~\ref{th:I} with $f\in
APW^+$ as in (\ref{2.23}), we will be able to complete the
consideration of $AP$ factorability in the big gap case.

It was shown earlier (see \cite[Chapter 14]{BKS1},
\cite[Theorem~2.2]{CaKaS09}) that $G$ is $APW$ factorable if, in
addition to the big gap requirement
$\eta_{1+}+\eta_{1-}\geq\lambda$, also
 \eq{S.A.1} \eta_{1+}\in \Omega (g_+)\text{ or }\eta_{1+}\geq\lambda, \quad
-\eta_{1-}\in\Omega (g_-)\text{ or }\eta_{1-}\geq\lambda.\en However,
the $AP$ factorability of $G$  if
$\lambda>\eta_{1+}\notin\Omega(g_+)$ or $\lambda>\eta_{1-}\notin
-\Omega(g_-)$  remained unsettled. As the next theorem shows, in
these cases $G$ does not have an $AP$ factorization.
\begin{thm}\label{th:S5.7}
Let $g\in APW$ be given by {\em (\ref{S.5.17A})}, with $\eta_{1\pm}$
defined by {\em (\ref{S5.18}), (\ref{S5.19})} and satisfying
$\eta_{1+}+\eta_{1-}\geq\lambda$.  Then the matrix function {\em
(\ref{4.1})} is $AP$ factorable if and only if {\em(\ref{S.A.1})} holds. In
this case $G$ actually admits an $APW$ factorization and its partial
indices are $\pm \mu$ with \eq{S.A.2} \mu=\min\{\lambda,\,
\eta_{1+},\, \eta_{1-},\, \eta_{1+}+\eta_{1-}-\lambda\}. \en In particular,
the factorization is canonical if and only if $\eta_{1+}=0$ or
$\eta_{1-}=0$ or $\eta_{1+}+\eta_{1-}=\lambda$.
\end{thm}

\begin{proof} {\sl Sufficiency}. Although it was established earlier, we
give here  a (much) shorter and self-contained proof, based on  the
results of Theorem \ref{th:K}. Namely, if (\ref{S.A.1}) is satisfied, then
(\ref{S5.31})--(\ref{S5.35}) hold with $0=\min \Omega(\tilde g_+)=\max
\Omega(\tilde
   g_-)$. Writing
\begin{eqnarray}
  \phi_+ &=& e_{\mu_1}\tilde \psi_+ \;\;\; \text{with}\;\;\; \mu_1=\min\{\lambda-\nu, -\nu+\eta_{1+}\}, \nonumber \\
  \phi_- &=&  e_{-\mu_2}\tilde \psi_- \;\;\; \text{with}\;\;\; \mu_2=\min\{\nu, \eta_{1-}+\nu-\lambda\}, \nonumber
\end{eqnarray}
we see that $\tilde\psi_\pm\in APW^\pm \cap CP^\pm$ and
$$Ge_{\mu_1+\mu_2}\tilde\psi_+=\tilde\psi_-,$$
so that, according to Theorem \ref{th:K}, $G$ admits an $APW$
factorization with partial indices $\pm \mu$ where
$$\mu=\mu_1+\mu_2=\min\{\lambda,\eta_{1+},\eta_{1-},\eta_{1+}+\eta_{1-}-\lambda\}$$
(as can be checked straightforwardly).

{\sl Necessity}. Suppose that
$\Omega(g_+)\not\ni\eta_{1+}<\lambda$; the case
$-\Omega(g_+)\not\ni\eta_{1-}<\lambda$ can be treated analogously.
Then a  solution to (\ref{RH}) with $\phi_\pm\in (APW^\pm)^2$ is
given by (\ref{S5.31})--(\ref{S5.35}).

It follows from these formulas that $\phi_{2+}=-e_{-\nu+\eta_{1+}}\tilde
g_+$, where $-\nu+\eta_{1+}\geq 0$  due to (\ref{S5.35}). On the other
hand, $0\notin \Omega(\tilde g_+)$ because $\eta_{1+}\notin
\Omega(g_+)$. Therefore, for any $\varepsilon>0$ and
$\nu=\eta_{1+}$ there is $y_\varepsilon\in\R^+$ such that
$$\inf_{\C^++iy_\varepsilon}|\phi_{2+}|=\inf_{\C^++iy_\varepsilon}|e_{-\nu+\eta_{1+}}\tilde g_+|<\varepsilon$$
and
$$\inf_{\C^++iy_\varepsilon}|\phi_{1+}|=\inf_{\C^++iy_\varepsilon}|e_{\lambda-\nu}|<\varepsilon.$$
Thus $\phi_+=(\phi_{1+},\phi_{2+})\notin CP^+$ and we conclude
from Theorem~\ref{th:cocri} that $G$ cannot have a canonical $AP$
factorization.

Now, if $G$ admits a non-canonical factorization, which must have
partial $AP$ indices $\pm \mu$ with $\mu>0$, then according to
Corollary~\ref{th:I}  we have (\ref{2.23}) with $f\in AP^+$, $\Omega (f)
\subset [0,\mu]$. Denoting $g_1^{\pm}=(g_{11}^\pm, g_{21}^\pm)$,
and considering in particular the first component of $\phi_+$, we thus
have from (\ref{S5.31}): \eq{S.A.7} e_{\lambda-\nu}=fg_{11}^+. \en In
addition, from the factorization it follows directly that \[
e_{-\lambda+\mu}g_{11}^+=g_{11}^-. \] Consequently, the
Bohr-Fourier spectrum of $g_{11}^+$ also is bounded, and (\ref{S.A.7})
therefore holds everywhere in $\C$.  In particular, $f$ and $g_{11}^+$
do not vanish in $\C$. But then (see \cite[Lemma 3.2]{C} or \cite[p.
371]{Le}) $\Omega(f)$, $\Omega(g_{11}^+)$ must each contain the
maximum and the minimum element, which implies that
\[ \max\Omega(f)+\max\Omega(g_{11}^+), \min\Omega(f)+\min\Omega(g_{11}^+)
\in\Omega(fg_{11}^+)=\{\lambda-\nu\}.\] We conclude that
$\min\Omega(f)=\max\Omega(f)$ and thus $f=e_\gamma$ for some
$\gamma\in [0,\mu]$.

But then, from  (\ref{S5.31})   and (\ref{2.23}), \[ (g_{11}^+, g_{21}^+)=
(e_{\lambda-\nu-\gamma}, e_{-\nu+\eta_{1+}-\gamma}\tilde g_+)\in
CP^+,
\] which is impossible when
$\Omega(g_+)\not\ni\eta_{1+}<\lambda$. Indeed, in this case
$\lambda-\nu-\gamma>\eta_{1+}-\nu-\gamma\geq 0$ and
$0\notin\Omega(\tilde g_+)$.

Finally, the criterion for the $AP$ factorization of $G$ to be canonical,
when it exists, follows immediately from formulas  (\ref{S.A.2}).
\end{proof}

\begin{rmk}\label{rmk:S5.8}
The last statement of  Theorem~{\em \ref{th:K}} implies that the
construction in the proof of Theorem~{\em \ref{th:S5.7}} delivers  not
only the partial $AP$ indices but also a first column of  $G_+$ and
$G_-$. Namely, they may be chosen equal  to $\tilde\psi_+$ and
$\tilde\psi_-$, respectively.
\end{rmk}

Now we move to the case $N>1$.

Knowing a solution (\ref{S.5.17J}), (\ref{S.5.17K}) of (\ref{RH}) and
using Theorem \ref{th:K} (with $\det G\equiv 1$, and therefore
$\delta=0$), we can obtain sufficient conditions for $AP$ factorability
of $G\in {\mathfrak{S}}_{{\lambda,N}}$, $N>1$.

\begin{thm}\label{th:S5.10}
Let $g\in APW$ be such that $g\in S_{{\lambda,N}}$, $N>1$, as
described in Theorem~{\em\ref{th:S5.5}}, with {\em(\ref{S5.30J})}
satisfied. Then $G$ admits an $APW$ factorization with partial $AP$
indices $\pm\mu$ where:
\begin{description}
\item[(i)] $\mu=N(\eta_{1+}+\eta_{1-})-\lambda$ if
\begin{equation}\label{x1}
\eta_{1+}\in\Omega(g_+),\;\; -\eta_{1-}\in\Omega(g_-)
\end{equation}
and
\begin{equation}\label{x2}
\lambda\ge\max\big\{N\eta_{1+}+(N-1)\eta_{2-},\,N\eta_{1-}+(N-1)\eta_{2+}\big\};
\end{equation}
\item[(ii)] $\mu=N\eta_{1+}-(N-1)\eta_{2+}$ if
\begin{equation}\label{x3}
\eta_{1+},\,\eta_{2+}\in\Omega(g_+)
\end{equation}
and
\begin{equation}\label{x4}
N\eta_{1+}+(N-1)\eta_{2-}\le\lambda\le N\eta_{1-}+(N-1)\eta_{2+};
\end{equation}
\item[(iii)] $\mu=N\eta_{1-}-(N-1)\eta_{2-}$ if
\begin{equation}\label{x5}
-\eta_{1-},\,-\eta_{2-}\in\Omega(g_-)
\end{equation}
and
\begin{equation}\label{x6}
N\eta_{1-}+(N-1)\eta_{2+}\le\lambda\le N\eta_{1+}+(N-1)\eta_{2-};
\end{equation}
\item[(iv)] $\mu=\lambda-(N-1)(\eta_{2+}+\eta_{2-})$ if
\begin{equation}\label{x7}
\eta_{2+}\in\Omega(g_+),\;\; -\eta_{2-}\in\Omega(g_-)
\end{equation}
and
\begin{equation}\label{x8}
\lambda\le\min\big\{N\eta_{1+}+(N-1)\eta_{2-},\,N\eta_{1-}+(N-1)\eta_{2+}\big\}.
\end{equation}
\end{description}
\end{thm}

\begin{proof}
Consider the solution to (\ref{RH}) given by (\ref{S.5.17J}),
(\ref{S.5.17K}), with $a_\pm, b_\pm$ as in (\ref{x9}), (\ref{x10}). Then
we obtain
\begin{align*}
\phi_{1+}&=e_{\lambda-N\nu}\sum_{j=0}^{N-1}\left((-1)^jg_-^jg_+^{N-1-j}\right) \\
&=\sum_{j=0}^{N-1}\left((-1)^je_{\lambda-N\nu-j\eta_{2-}+(N-1-j)\eta_{1+}}
(e_{\eta_{2+}}g_-)^j(e_{-\eta_{1+}}g_+)^{N-1-j}\right) \\
&=e_{\lambda-N\nu-(N-1)\eta_{2-}}\widetilde\phi_{1+},
\end{align*}
with $\widetilde\phi_{1+}\in APW^+$ where
$\lambda-N\nu-(N-1)\eta_{2-}\ge 0$ due to (\ref{S5.25}) and
$0=\inf\Omega(\widetilde\phi_{1+})$
($=\min\Omega(\widetilde\phi_{1+})$\, if\,
$-\eta_{2-}\in\Omega(g_-)$);
\[
\phi_{2+}=-e_{-N\nu}g_+^{N}=-e_{-N\nu+N\eta_{1+}}(e_{-\eta_{1+}}g_+)^N
=e_{-N\nu+N\eta_{1+}}\widetilde\phi_{2+},
\]
with $\widetilde\phi_{2+}\in APW^+$ where $-N\nu+N\eta_{1+}\ge 0$
due to (\ref{S5.25}) and $0=\inf\Omega(\widetilde\phi_{2+})$
($=\min\Omega(\widetilde\phi_{2+})$\, if\,
$\eta_{1+}\in\Omega(g_+)$);
\begin{align*}
\phi_{1-}&=e_{-N\nu}\sum_{j=0}^{N-1}\left((-1)^jg_-^jg_+^{N-1-j}\right) \\
&=\sum_{j=0}^{N-1}\left( (-1)^je_{-N\nu-j\eta_{1-}+(N-1-j)\eta_{2+}}
(e_{\eta_{1-}}g_-)^j(e_{-\eta_{2+}}g_+)^{N-1-j}\right) \\
&=e_{-N\nu+(N-1)\eta_{2+}}\widetilde\phi_{1-},
\end{align*}
with $\widetilde\phi_{1-}\in APW^-$ where $-N\nu+(N-1)\eta_{2+}\le
0$ due to (\ref{S5.25}) and $0=\sup\Omega(\widetilde\phi_{1-})$
($=\max\Omega(\widetilde\phi_{1-})$\, if\,
$\eta_{2+}\in\Omega(g_+)$);
\[
\phi_{2-}=(-1)^{N-1}e_{\lambda-N\nu}g_-^{N}=(-1)^{N-1}e_{\lambda-N\nu-N
\eta_{1-}}(e_{\eta_{1-}}g_-)^N
=e_{\lambda-N\nu-N\eta_{1-}}\widetilde\phi_{2-},
\]
with $\widetilde\phi_{2-}\in APW^-$ where
$\lambda-N\nu+N\eta_{1-}\ge 0$ due to (\ref{S5.25}) and
$0=\sup\Omega(\widetilde\phi_{2-})$
($=\max\Omega(\widetilde\phi_{2-})$\, if\,
$-\eta_{1-}\in\Omega(g_-)$).

Hence,
\begin{equation}\label{x11}
G\begin{bmatrix}e_{\lambda-N\nu-(N-1)\eta_{2-}}\widetilde\phi_{1+}\\[1mm]
e_{-N\nu+N\eta_{1+}}\widetilde\phi_{2+}\end{bmatrix}
=\begin{bmatrix}e_{-N\nu+(N-1)\eta_{2+}}\widetilde\phi_{1-}\\[1mm]
e_{\lambda-N\nu-N\eta_{1-}}\widetilde\phi_{2-}\end{bmatrix}.
\end{equation}
Setting now $\phi_+=e_{\mu_1}\widetilde\psi_+$ and
$\phi_-=e_{-\mu_2}\psi_-$ where
\[
\begin{aligned}
\mu_1&=-N\nu+\min\big\{\lambda-(N-1)\eta_{2-},\,N\eta_{1+}\big\}\ge 0,\\[1mm]
\mu_2&=N\nu+\min\big\{-(N-1)\eta_{2+},\,N\eta_{1-}-\lambda\big\}\ge 0,
\end{aligned}
\]
we infer from (\ref{x11}) that $G\psi_+=\psi_-$, with
$\psi_+=e_{\mu}\widetilde\psi_+$ and
\begin{align}\label{x121}
\mu &=\mu_1+\mu_2=\min\big\{\lambda-(N-1)\eta_{2-},\,N\eta_{1+}\big\}+
\min\big\{-(N-1)\eta_{2+},\,N\eta_{1-}-\lambda\big\}\notag\\
&=\min\big\{N(\eta_{1+}+\eta_{1-})-\lambda,\:N\eta_{1+}-(N-1)\eta_{2+},\notag\\
&\qquad\qquad N\eta_{1-}-(N-1)\eta_{2-},\:\lambda-(N-1)(\eta_{2+}+\eta_{2-})\big\}.
\end{align}

We consider separately the cases (i)-(iv).

(i) If (\ref{x1}) and (\ref{x2}) hold, then
$\mu=N(\eta_{1+}+\eta_{1-})-\lambda$ due to (\ref{x121}) and
\[
\widetilde\psi_+=\begin{bmatrix}e_{\lambda-N\eta_{1+}-(N-1)\eta_{2-}}
\widetilde\phi_{1+}\\[1mm]\widetilde\phi_{2+}\end{bmatrix},\quad
\psi_-=\begin{bmatrix}e_{-\lambda+N\eta_{1-}+(N-1)\eta_{2+}}
\widetilde\phi_{1-}\\[1mm]\widetilde\phi_{2-}\end{bmatrix}
\]
where $M(\widetilde\phi_{2+})\ne 0$ if and only
$\eta_{1+}\in\Omega(g_+)$, and $M(\widetilde\phi_{2-})\ne 0$ if and
only $-\eta_{1-}\in\Omega(g_-)$. Hence, by (\ref{x1}),
$\widetilde\psi_+=e_{-\mu}\psi_+\in CP^+$ and $\psi_-\in CP^-$. The
result now follows from Theorem \ref{th:K}.

(ii) If (\ref{x3}) and (\ref{x4}) hold, then
$\mu=N\eta_{1+}-(N-1)\eta_{2+}$ due to (\ref{x121}) and
\[
\widetilde\psi_+=\begin{bmatrix}e_{\lambda-N\eta_{1+}-(N-1)\eta_{2-}}
\widetilde\phi_{1+}\\[1mm]\widetilde\phi_{2+}\end{bmatrix},\quad
\psi_-=\begin{bmatrix}\widetilde\phi_{1-}\\[1mm]
e_{\lambda-N\eta_{1-}-(N-1)\eta_{2+}}\widetilde\phi_{2-}\end{bmatrix}
\]
where $M(\widetilde\phi_{2+})\ne 0$ if and only
$\eta_{1+}\in\Omega(g_+)$, and $M(\widetilde\phi_{1-})\ne 0$ if and
only $\eta_{2+}\in\Omega(g_+)$. Hence, by (\ref{x3}),
$\widetilde\psi_+=e_{-\mu}\psi_+\in CP^+$ and $\psi_-\in CP^-$. The
result now follows from Theorem \ref{th:K}.

(iii) If (\ref{x5}) and (\ref{x6}) hold, then
$\mu=N\eta_{1-}-(N-1)\eta_{2-}$ due to (\ref{x121}) and
\[
\widetilde\psi_+=\begin{bmatrix}\widetilde\phi_{1+}\\[1mm]
e_{-\lambda+N\eta_{1+}+(N-1)\eta_{2-}}\widetilde\phi_{2+}\end{bmatrix},\quad
\psi_-=\begin{bmatrix}e_{-\lambda+N\eta_{1-}+(N-1)\eta_{2+}}
\widetilde\phi_{1-}\\[1mm]\widetilde\phi_{2-}\end{bmatrix}
\]
where $M(\widetilde\phi_{1+})\ne 0$ if and only
$-\eta_{2-}\in\Omega(g_-)$, and $M(\widetilde\phi_{2-})\ne 0$ if and
only $-\eta_{1-}\in\Omega(g_-)$. Hence, by (\ref{x5}),
$\widetilde\psi_+=e_{-\mu}\psi_+\in CP^+$ and $\psi_-\in CP^-$. The
result now follows from Theorem \ref{th:K}.

(iv) If (\ref{x7}) and (\ref{x8}) hold, then
$\mu=\lambda-(N-1)(\eta_{2+}+\eta_{2-})$ due to (\ref{x121}) and
\[
\widetilde\psi_+=\begin{bmatrix}\widetilde\phi_{1+}\\[1mm]
e_{-\lambda+N\eta_{1+}+(N-1)\eta_{2-}}\widetilde\phi_{2+}\end{bmatrix},\quad
\psi_-=\begin{bmatrix}\widetilde\phi_{1-}\\[1mm]
e_{\lambda-N\eta_{1-}-(N-1)\eta_{2+}}\widetilde\phi_{2-}\end{bmatrix}
\]
where $M(\widetilde\phi_{1+})\ne 0$ if and only
$-\eta_{2-}\in\Omega(g_-)$, and $M(\widetilde\phi_{1-})\ne 0$ if and
only $\eta_{2+}\in\Omega(g_+)$. Hence, by (\ref{x7}),
$\widetilde\psi_+=e_{-\mu}\psi_+\in CP^+$ and $\psi_-\in CP^-$. The
result again follows from Theorem \ref{th:K}.
\end{proof}

\begin{rmk}\label{rmk:S5.11}
If $\lambda=N\eta_{1+}+(N-1)\eta_{2-}=N\eta_{1-}+(N-1)\eta_{2+}$,
then all the numbers
\begin{gather*}
N(\eta_{1+}+\eta_{1-})-\lambda,\quad N\eta_{1+}-(N-1)\eta_{2+},\\
N\eta_{1-}-(N-1)\eta_{2-},\quad\lambda-(N-1)(\eta_{2+}+\eta_{2-})
\end{gather*}
coincide, and therefore $\mu$ in Theorem~{\em\ref{th:S5.10}} is equal
to their common value. Analogously, if
$\lambda=N\eta_{1+}+(N-1)\eta_{2-}$, then
\begin{align*}
N(\eta_{1+}+\eta_{1-})-\lambda &=N\eta_{1-}-(N-1)\eta_{2-},\\
\lambda-(N-1)(\eta_{2+}+\eta_{2-}) &=N\eta_{1+}-(N-1)\eta_{2+},
\end{align*}
and if $\lambda=N\eta_{1-}+(N-1)\eta_{2+}$, then
\begin{align*}
N(\eta_{1+}+\eta_{1-})-\lambda &=N\eta_{1+}-(N-1)\eta_{2+},\\
\lambda-(N-1)(\eta_{2+}+\eta_{2-}) &=N\eta_{1-}-(N-1)\eta_{2-}.
\end{align*}
Hence, in the latter two cases
$\mu=\min\big\{N(\eta_{1+}+\eta_{1-})-\lambda,\:
\lambda-(N-1)(\eta_{2+}+\eta_{2-})\big\}$.
\end{rmk}

\begin{rmk}\label{rmk:S5.12}
The main difficulty in applying Theorem~{\em\ref{th:S5.10}} lies in
verifying whether or not condition {\em(\ref{S5.30J})} holds. In this
regard, Theorems {\rm 3.1} and {\rm 3.4} of {\em \cite{CaDi08}} may
be helpful. Also, as was mentioned before, {\em(\ref{S5.30J})} holds if
$a_+$ or $a_-$ is a single exponential. A class of matrix functions with
such $a_\pm$ was studied in {\em \cite{CaKaS09}}, where the $APW$
factorization of $G$ was explicitly obtained. Naturally, conclusions of
{\em \cite{CaKaS09}} match those that can be obtained by applying
Theorem~{\em\ref{th:S5.10}} to the same class. Furthermore,
combining Corollary~{\em \ref{th:I}} and Theorem~{\em\ref{th:J}} of
the present paper with the $APW$ factorization obtained in {\em
\cite{CaKaS09}}, it is possible to characterize completely the solutions
of {\em(\ref{RH})} in that case.
\end{rmk}

Below we give examples of two cases in which condition
(\ref{S5.30J}) is also not hard to verify.

\begin{ex}\label{ex:S5.13}
\emph{Let the off-diagonal entry $g\in S_{{\lambda,N}}$ of the matrix
(\ref{4.1}) be given by}
\[
    g=c_{-2}e_{-\eta_{2-}}+c_{-1}e_{-\eta_{1-}}+g_+
\]
\emph{with} $c_{-2}, c_{-1}\in \C\,,\, 0\leq\eta_{1-}<\eta_{2-}$
\emph{and} $g_+\in APW^+$ \emph{with Bohr-Fourier spectrum
containing its maximum and minimum points $\eta_{j+}$, $j=1,2$.}

\emph{If $N=1$, which happens in particular if $c_{-1}=c_{-2}=0$, then
$G$ is $APW$ factorable with partial $AP$ indices given by Theorem
\ref{th:S5.7}.}

\emph{If $N>1$, then it follows from Theorem \ref{th:S5.10} that $G$
admits an $APW$ factorization with partial $AP$ indices $\pm\mu$,
where }
\[
\mu=\!\begin{cases}N(\eta_{1+}+\eta_{1-})-\lambda\;\;\text{if}\;\;\;
\lambda\ge\max\big\{N\eta_{1+}+(N-1)\eta_{2-},\,N\eta_{1-}+(N-1)\eta_{2+}\big\},\\
N\eta_{1+}-(N-1)\eta_{2+}\;\;\text{if}\;\;\;
N\eta_{1+}+(N-1)\eta_{2-}\le\lambda\le N\eta_{1-}+(N-1)\eta_{2+},\\
N\eta_{1-}-(N-1)\eta_{2-}\;\;\text{if}\;\;\;
N\eta_{1-}+(N-1)\eta_{2+}\le\lambda\le N\eta_{1+}+(N-1)\eta_{2-},\\
\lambda\!-\!(N-1)(\eta_{2+}+\eta_{2-})\;\;\text{if}\;\;
\lambda\le\min\!\big\{\!N\eta_{1+}\!+\!(N-1)\eta_{2-},\,N\eta_{1-}\!+\!(N-1)\eta_{2+}\!\big\},
\end{cases}
\]
\emph{whenever (\ref{S5.30J}) holds. Moreover, the expressions
given in the proof of Theorem \ref{th:S5.10} for
$\phi_{1\pm},\phi_{2\pm}$ in each case also provide, by using
Theorem \ref{th:K}, one column for the factors $G_\pm$ in an $APW$
factorization of $G$.}

{\em In its turn, condition (\ref{S5.30J}) is satisfied if and only if one of
the coefficients $c_{-1}, c_{-2}$ is zero or (if $c_{-1}c_{-2}\neq 0$)}
\begin{equation}\label{E2}
    \inf_{k\in\Z}|g_+(z_k)|>0
\end{equation}
\emph{where }$z_k$, $k\in\Z$\emph{, are the zeros of}
$g_-=c_{-2}e_{-\eta_{2-}}+c_{-1}e_{-\eta_{1-}}$,\emph{ i.e.,}
\[ 
    z_k=\frac{1}{\eta_{2-}-\eta_{1-}}\left(\arg \left(-\frac{c_{-2}}{c_{-1}}\right)+2k\pi -i\log
    \left|\frac{c_{-2}}{c_{-1}}\right|\right).
\] 

\emph{If, in particular, $g_+$ also is a binomial, i.e.,}
\[ 
g_+=c_{1}e_{\eta_{1+}}+c_{2}e_{\eta_{2+}}\hspace{0.5cm} (c_{1},
c_{2}\in \C\,, \, 0\leq\eta_{1+}<\eta_{2+})
\] 
\emph{then (\ref{E2}) is satisfied whenever one of the coefficients
$c_1, c_2$ is zero. On the other hand, for $c_1, c_2\neq 0$ condition
(\ref{E2}) is equivalent to (cf. Lemma~3.3 in \cite{BKdST2})}
\[ 
    \left|\frac{c_1}{c_2}\right|^{\eta_{2-}-\eta_{1-}}\neq
    \left|\frac{c_{-2}}{c_{-1}}\right|^{\eta_{2+}-\eta_{1+}}
    \hspace{0.4cm}\text{\emph{if}} \hspace{0.4cm}
    \frac{\eta_{2+}-\eta_{1+}}{\eta_{2-}-\eta_{1-}}\in\R\backslash\Q;
\] 
and to
\[ 
    \left(-\frac{c_1}{c_2}\right)^{q}\neq
    \left(-\frac{c_{-2}}{c_{-1}}\right)^{p}
    \hspace{0.4cm}\text{\emph{if}} \hspace{0.4cm} \frac{\eta_{2+}-\eta_{1+}}{\eta_{2-}-\eta_{1-}}=\frac{p}{q},
    \hspace{0.4cm}\text{ \emph{with }} \hspace{0.4cm} p, q\in \N \hspace{0.4cm} \text{\emph{relatively
    prime}}.
\] 
\end{ex}

\begin{ex}\label{ex:S5.14}
\emph{Let $G\in {{\mathfrak{S}}}_{{\lambda,N}}$, $N>1$, with the
off-diagonal entry $g\in APW$ of the form $g=g_-+g_+$ where}
\begin{equation}\nonumber
g_+=c_\alpha e_\alpha g_-+c_\mu e_\mu,
\end{equation}
\emph{$\alpha, \mu>0$, $c_\alpha, c_\mu\in \C, c_\mu\neq 0$ and
$\eta_{1\pm}, \eta_{2\pm}\in \Omega(g_\pm)$ (see (\ref{S5.18}),
(\ref{S5.19})). It is easy to see that (\ref{S5.30J}) holds.
Theorem~\ref{th:S5.10} implies therefore that  $G$ admits an
    $APW$ factorization with partial $AP$ indices as indicated in
    that theorem.}
\end{ex}

\newpage

%
\end{document}